\documentclass[a4paper,11pt]{amsart}

\usepackage[centertags]{amsmath}
\usepackage[utf8]{inputenc}
\usepackage{amssymb,amsthm,dsfont,mathrsfs,upgreek,tikz-cd,graphicx,fp,color,exscale,anysize,url,enumerate}
\marginsize{2cm}{2cm}{2cm}{2cm}
\usepackage[all]{xy}
\usepackage{hyperref}
\usepackage{cite}

{\theoremstyle{definition}
\newtheorem{definition}{Definition}[section]

\newtheorem{defn}[definition]{Definition}
\newtheorem{defi}[definition]{Definition}}
\newtheorem{theorem}[definition]{Theorem}
\newtheorem{proposition}[definition]{Proposition}
\newtheorem{lemma}[definition]{Lemma}
\newtheorem{corollary}[definition]{Corollary}

\newtheorem{thm}[definition]{Theorem}
\newtheorem{prop}[definition]{Proposition}

{\theoremstyle{remark}
\newtheorem{rem}[definition]{Remark}
\newtheorem{remark}[definition]{Remark}
}

\newcommand{\dR}{{\mathds R}}
\newcommand{\dC}{{\mathds C}}

\newcommand{\dN}{{\mathds N}}

\newcommand{\dZ}{{\mathds Z}}
\newcommand{\dP}{{\mathds P}}

\newcommand{\dL}{{\mathds L}}
\newcommand{\dA}{{\mathbb A}}

\newcommand{\dG}{{\mathds G}}

\newcommand{\cD}{\mathcal{D}}

\newcommand{\cH}{\mathcal{H}}
\newcommand{\cI}{\mathcal{I}}
\newcommand{\cJ}{\mathcal{J}}
\newcommand{\cK}{\mathcal{K}}

\newcommand{\cM}{\mathcal{M}}
\newcommand{\cN}{\mathcal{N}}
\newcommand{\cO}{\mathcal{O}}
\newcommand{\cP}{\mathcal{P}}

\newcommand{\cR}{\mathcal{R}}

\newcommand{\cT}{\mathcal{T}}

\newcommand{\cX}{\mathcal{X}}

\newcommand{\fA}{\mathfrak{A}}

\newcommand{\fe}{\mathfrak{e}}

\newcommand{\fp}{\mathfrak{p}}

\newcommand{\bL}{\mathbf L}
\newcommand{\bR}{\mathbf R}

\newcommand{\D}{\displaystyle}

\DeclareMathOperator{\For}{\textup{For}}
\DeclareMathOperator{\Spec}{\textup{Spec}\,}

\DeclareMathOperator{\diag}{\textup{diag}}

\newcommand{\Hom}{\operatorname{Hom}}

\newcommand{\DR}{\operatorname{DR}}

\newcommand{\FL}{\operatorname{FL}}
\DeclareMathOperator{\id}{\textup{id}}
\DeclareMathOperator{\Gr}{\textup{Gr}}

\newcommand{\wh}{\widehat}
\newcommand{\dd}{\partial}

\newcommand{\ep}{\varepsilon}
\newcommand{\ul}{\underline}
\newcommand{\ttau}{{}^\tau\!}
\newcommand{\ttheta}{{}^\theta\!}
\newcommand{\inv}{\operatorname{inv}}
\newcommand{\Firr}{F^{\operatorname{irr}}}

\newcommand{\mba}{\mathds{A}}
\newcommand{\mbc}{\mathds{C}}

\newcommand{\mbg}{\mathds{G}}
\newcommand{\mbl}{\mathds{L}}
\newcommand{\mbn}{\mathds{N}}
\newcommand{\mbp}{\mathds{P}}

\newcommand{\mbr}{\mathds{R}}

\newcommand{\mbz}{\mathds{Z}}

\newcommand{\mca}{\mathcal{A}}

\newcommand{\mcd}{\mathcal{D}}
\newcommand{\mce}{\mathcal{E}}

\newcommand{\mch}{\mathcal{H}}
\newcommand{\mci}{\mathcal{I}}
\newcommand{\mcj}{\mathcal{J}}

\newcommand{\mcm}{\mathcal{M}}
\newcommand{\mcn}{\mathcal{N}}
\newcommand{\mco}{\mathcal{O}}

\newcommand{\mcr}{\mathcal{R}}

\newcommand{\mct}{\mathcal{T}}

\newcommand{\msd}{\mathscr{D}}
\newcommand{\msh}{\mathscr{H}}
\newcommand{\msi}{\mathscr{I}}
\newcommand{\msj}{\mathscr{J}}

\newcommand{\mso}{\mathscr{O}}

\newcommand{\msr}{\mathscr{R}}

\newcommand{\msx}{\mathscr{X}}
\newcommand{\msy}{\mathscr{Y}}

\newcommand{\kaa}{\mca^{\varphi/z}_{\text{aff}}}
\newcommand{\kap}{\mca^{\varphi/z}_{*}}
\newcommand{\khn}{\mce^{\varphi/z}_{*}}
\newcommand{\kh}{\mce^{\varphi/z}}
\newcommand{\kmtm}{\mct^{\varphi/z}}

\newcommand{\laa}{\mca^{\psi/z}_{\text{aff}}}
\newcommand{\lap}{\mca^{\psi/z}_{*}}
\newcommand{\lhn}{\mce^{\psi/z}_{*}}
\newcommand{\lh}{\mce^{\psi/z}}
\newcommand{\lmtm}{\mct^{\psi/z}}

\newcommand{\MTM}{\operatorname{MTM}}
\newcommand{\MHM}{\operatorname{MHM}}
\newcommand{\cRint}{\cR^{\operatorname{int}}}
\newcommand{\IrrMHM}{\operatorname{IrrMHM}}

\newcommand{\ra}{\rightarrow}
\newcommand{\hra}{\hookrightarrow}
\newcommand{\lra}{\longrightarrow}

\newcommand{\p}{\partial}

\newsavebox\foobox

\begin{document}
\title{Examples of hypergeometric twistor $\cD$-modules}
\author{Alberto Castaño Domínguez}
\author{Thomas Reichelt}
\author{Christian Sevenheck}
\thanks{The authors are partially supported by the project SISYPH: ANR-13-IS01-0001-01/02 and DFG grant SE 1114/5-1. The first named author is also partially supported by MTM2014-59456-P, the ERDF and GI-2136. The second named author is supported by a DFG Emmy-Noether-Fellowship (RE 3567/1-1).}
\date{}
\email{alberto.castano@usc.es}
\email{treichelt@mathi.uni-heidelberg.de}
\email{christian.sevenheck@mathematik.tu-chemnitz.de}
\address{Instituto de Matemáticas, Universidade de Santiago de Compostela, 15782 Santiago de Compostela (Spain); Fakultät für Mathematik, Technische Universität Chemnitz, 09107 Chemnitz (Germany)}
\address{Mathematisches Institut, Universit\"at Heidelberg, Im Neuenheimer Feld 205, 69120 Heidelberg (Germany)}
\address{Fakultät für Mathematik, Technische Universität Chemnitz, 09107 Chemnitz (Germany)}
\subjclass[2010]{Primary 14F10, 32C38}
\keywords{D-modules, irregular Hodge filtration, twistor D-modules, Fourier-Laplace transformation, hypergeometric D-modules}

\begin{abstract}
We show that certain one-dimensional hypergeometric differential systems
underlie objects of the category of irregular mixed Hodge modules, which was recently
introduced by Sabbah, and compute the irregular Hodge filtration for them. We also provide a comparison theorem between two different types of
Fourier-Laplace transformation for algebraic integrable twistor $\cD$-modules.
\end{abstract}

\maketitle

\section{Introduction}

\label{sec:Introduction}

In a series of papers (see \cite{Yu,ESY,SabYu,Sa15}), Sabbah and Yu (partly joint with Esnault) have considered a so-called irregular Hodge
filtration on certain cohomology groups resp. on certain irregular $\cD$-modules. It can be seen as a generalization of
the Hodge filtration on a mixed Hodge module in the sense of M.~Saito.
Geometrically, such a filtration arises by considering a version of the twisted de Rham cohomology of certain proper maps,
and it plays (conjecturally) a role in Hodge theoretic mirror symmetry (see \cite{KKP2}). In \cite{Sa15}, Sabbah has defined a category of irregular mixed Hodge modules, which is (up to a technical equivalence) a certain subcategory of T.~Mochizuki's category of (integrable) mixed twistor $\cD$-modules.
He has proved that a rigid irreducible $\cD$-module on the projective line can be uniquely upgraded to an irregular Hodge module
if and only if its formal local monodromies are unitary. Consequently, these objects come equipped with an irregular Hodge filtration
and one can define irregular Hodge numbers for them. They should be seen as interesting numerical invariants attached to
these differential systems, contrary to the case of arbitrary mixed twistor $\cD$-modules, where there is no obvious way to define such numbers.
In \cite{CDS}, the first and the third named author have computed that filtration and its corresponding numbers for the purely irregular
hypergeometric modules, that is for systems of the form $\cD_{\dG_m}/\cD_{\dG_m}P$, where $P$ is the operator
$$
P=\prod_{i=1}^{n} \left(t\partial_t-\alpha_i\right)-t
$$
for real numbers $\alpha_1,\ldots,\alpha_n$. Let us consider the non-commutative ring $R^{\text{int}}_{\dG_m}:=\dC[z,t^\pm]\langle z^2\partial_z, tz\partial_t \rangle$. A crucial point was to show that a certain quotient of the corresponding sheaf $\cRint_{\dG_m}$ on $\dG_m$ which restricts to the $\cD_{\dG_{m,t}}$-module $\cD_{\dG_m}/\cD_{\dG_m}P$ on $z=1$ actually underlies an object in the category $\IrrMHM(\dG_m)$ and the latter can be uniquely extended to an object in $\IrrMHM(\dP^1)$.

In this paper we discuss the case of more general hypergeometric $\cD$-module, that is, for quotients $\cD_{\dG_m}/\cD_{\dG_m}P$, where now $P$ is of the form
$$
P=\prod_{i=1}^{n} \left(t\partial_t-\alpha_i\right)-t \prod_{j=1}^{m} \left(t\partial_t-\beta_j\right)
$$
for positive integers $m,n$ and real numbers $\alpha_1,\ldots,\alpha_n,\beta_1,\ldots,\beta_m$ such that there is no integer difference
between any $\alpha_i$ and $\beta_j$ (this is the irreducibility assumption). It is worth noticing that the presence of the factor $\prod_{j=1}^{m}(t\partial_t-\beta_j)$
rules out the usage of the geometric arguments of \cite{CDS}. We obtain (see Theorem \ref{thm:1dimhypgeo})
that for certain such systems, the corresponding quotient of $\cRint_{\dG_m}$ still underlies
an object of $\IrrMHM(\dG_m)$. As an application, we can completely determine the irregular Hodge filtration
for all systems $\cD/\cD P$ as above, where $n$ is arbitrary and where $m=1$.

The strategy of the proof (which is rather different from that of \cite{CDS}) of the main theorem is to reduce these differential systems from (Fourier-Laplace transformed) A-hypergeometric $\cD$-modules
(the so-called GKZ-systems of Gel'fand, Graev, Zelevinski and Kapranov, see \cite{GGZ87},\cite{GKZ89}), but at the level of (algebraic, integrable, mixed) twistor
$\cD$-modules. We can use a central result of \cite{RS15}, where the Hodge filtration on certain of
these GKZ-systems has been computed explicitly. Technically, the main point in our proof consists in showing that for an $\cR$-module underlying an integrable mixed twistor $\cD$-module on the affine space, the algebraic Fourier-Laplace transformation (which is defined very much the same as in the case of algebraic $\cD$-modules)
coincides with the Fourier-Laplace transformation that can be defined inside the category $\MTM$, or even $\IrrMHM$.
Along the way, we also obtain that an $\cR$-module version of the GKZ-$\cD$-module underlies an irregular Hodge module.

Our results give concrete representations for objects in the category $\MTM$ resp. $\IrrMHM$, which usually are difficult to describe explicitly.
We hope that a similar approach can be used to understand the irregular Hodge filtration for some higher dimensional analogues of the classical hypergeometric systems, also called Horn systems, which occur in the mirror symmetry picture for toric varieties.\\

\textbf{Acknowledgements.} We would like to thank Takuro Mochizuki for communicating the proof of Lemma \ref{lem:Moch} to us. We would further like to thank Takuro Mochizuki and Claude Sabbah for their interest in our work and for many stimulating discussions. We are grateful to the participants of the workshop \textit{Mixed Twistor D-modules} in Heidelberg in 2017 for their time and effort. We also thank the Max Planck Institute for Mathematics in the Sciences, where  some part of the work presented here has been carried out.

\section{Some results on $\msr$- and mixed twistor $\cD$-modules}

Let $X$ be a complex manifold of dimension $d$.  We denote by $\mso_X$ the sheaf of holomorphic functions and  $\msd_X$ the sheaf of differential operators with holomorphic coefficients. Recall that $\msd_X$ is generated by the tangent sheaf $\Theta_X$.
We put $\msx := \dA^1_z \times X$, where the subscript means that $z$ is the canonical coordinate on $\dA^1$. Denote by $p_z: \msx \ra X$ the projection.  We denote by $\msr_\msx$ the sheaf of subalgebras of $\msd_\msx$ generated by $z p_z^* \Theta_X$ over $\mso_\msx$ and by $\msr^{\operatorname{int}}_\msx$ the sheaf of subalgebras of $\msd_\msx$ generated by $z p_z^* \Theta_X$ and $z^2 \p_z$ over $\mso_\msx$. In local coordinates $x_1,\ldots,x_d$, they are given by $\mso_\msx\langle z \p_{x_1},\ldots, z \p_{x_d}\rangle$ and $\mso_\msx\langle z^2\p_z, z \p_{x_1},\ldots, z \p_{x_d}\rangle$, respectively. We set $\Omega^1_\msx := z^{-1}p^*_z \Omega^1_X$ as a subsheaf of $p_z^*\Omega^1_X \otimes \mso_\msx(*(\{0\} \times X))$, $\Omega^p_\msx := \bigwedge^p \Omega^1_\msx$ and $\omega_\msx := \Omega^d_\msx$.

Let $f:X \ra Y$  be a morphism of complex manifolds. We consider the transfer $\msr$-modules, given by $\msr_{\msx \ra\msy} := \mso_{\msx} \otimes_{f^{-1} \mso_\msy} f^{-1} \msr_\msy$ and $\msr_{\msy \leftarrow \msx} := \omega_{\msx} \otimes \msr_{\msx \ra \msy} \otimes f^{-1}\omega_\msy$, being respectively a $(\msr_\msx,f^{-1}\msr_\msy)$-bimodule and a $(f^{-1} \msr_\msy,\msr_\msx)$-bimodule. We have the inverse image and direct image functors
\begin{equation}\label{eq:ImageFunctors}\begin{split}
f^+(\mcn) := \msr_{\msx \ra \msy} \overset{\bL}\otimes_{f^{-1}\msr_\msy} f^{-1} \mcn,\\
f_+(\mcm) := \bR f_*(\msr_{\msy \leftarrow \msx} \overset{\bL}\otimes_{\msr_\msx} \mcm ),
\end{split}\end{equation}
between the bounded derived categories $D^b(\msr_\msx)$ and $D^b(\msr_\msy)$.

If $f:X \times Y \ra Y$ is a projection and $\dim X=d$, then $f_+(\mcm)$ is given by
\[
f_+(\mcm) = \bR f_* \operatorname{DR}_{\msx \times \msy / \msy} (\mcm)[d],
\]
where  $\DR_{\msx \times \msy / \msy} (\mcm)$ is the relative de Rham complex with differential
\[
d(\eta \otimes m) = d\eta \otimes m + \sum_{i=1}^d \left(\frac{dx_i}{z} \wedge \eta \right)\otimes z\p_{x_i}m,
\]
the $(x_i)_{1 \leq i \leq d}$ being local coordinates on $X$.

Let $\sigma: \dG_{m,z} \ra \dG_{m,z}$ be the automorphism $z \mapsto - \overline{z}^{-1}$. Set $\mathbf{S}:= \{z \in \dA^1_z \mid |z| =1\}$. If $\lambda \in \mathbf{S}$ then $\sigma(\lambda) = - \lambda$. Let $\mce^{(d,d)}_{\mathbf{S} \times X/\mathbf{S},c}(V)$ the space of $C^\infty$-sections of $\Omega^{d,d}_{\mathbf{S} \times X / \mathbf{S}}$ over any open subset $V$ of $\mathbf{S} \times X$ with compact support and $C^0_c(\mathbf{S})$ the space of continous functions on $\mathbf{S}$ with compact support.  The space of $C^\infty(\mathbf{S})$-linear maps $\mce^{(n,n)}_{\mathbf{S} \times X / \mathbf{S},c}(V) \ra C^0_c(\mathbf{S})$ is denoted by $\mathfrak{D}\mathfrak{b}_{\mathbf{S} \times X / \mathbf{S}}(V)$. This gives rise to the sheaf $\mathfrak{D}\mathfrak{b}_{\mathbf{S} \times X / \mathbf{S}}$. The abelian category $\msr$-Tri(X) consists of triples $\mct= (\mcm_1,\mcm_2,C)$ where $\mcm_1,\mcm_2$ are $\msr_\msx$-modules and $C: \mcm_{1 \mid \mathbf{S} \times X} \otimes \sigma^* \mcm_{2 \mid \mathbf{S} \times X} \ra \mathfrak{D} \mathfrak{b}_{\mathbf{S} \times X / \mathbf{S}}$ is a $\msr_{\msx \mid \mathbf{S} \times X} \otimes \sigma^* \msr_{\msx \mid \mathbf{S} \times \msx}$-linear morphism. If $D \subset X$ is a hypersurface,
one similarly defines a category $\msr$-Tri$(X,D)$ using $\msr_{\msx}(*D):= \msr_\msx \otimes_{\mso_\msx} \mso_{\msx}(*(\dA^1_z \times D))$-modules (cf. \cite[\S~2.1]{Mo13} for details).

Now let $X:= X_0 \times \dA^1_t$ and let $\Theta_X(\log X_0)$ be the sheaf of vector fields on $X$ which are logarithmic along $X_0$. Let $V_0 \msr_\msx$ be the sheaf of sub-algebras in $\msr_\msx$  which is generated by $z p^*_z \Theta_X(\log X_0)$.  For $z_0 \in \dA^1_z$ we denote by $\msx^{(z_0)}$ a small neighborhood of $\{z_0\} \times X$. A coherent $\msr_\msx$-module is called strictly specializable along $t$ at $z_0$ if $\mcm_{\mid \msx^{(z_0)}}$ is equipped with an increasing and exhaustive filtration $V_a^{(z_0)}(\mcm_{\mid \msx^{(z_0)}})_{a \in \mbr}$ by coherent $(V_0 \msr_{\msx})_{\mid\msx^{(z_0)}}$-modules satisfying certain conditions (cf. \cite[\S\S~2.1.2.1, 2.1.2.2]{Mo13}). This filtration is unique if it exists.  $\mcm$ is called strictly specializable along $t$ if it is strictly specializable along $t$ for any $z_0$.

\begin{remark}\label{rem:CoherSpecAlongT}
If $\mcm$ is itself a coherent $V_0 \msr_\msx$-module, then $\mcm$ is automatically specializable along $t$ and the corresponding filtration $V_a(\mcm)$ exists globally and is trivial, i.e. $V_a(\mcm) = V_b(\mcm)$ for all $a,b \in \mbr$.
\end{remark}

If $\mcm$ is a coherent $\msr_\msx(*t)$-module, we define similarly a filtration $V^{(z_0)}_a(\mcm_{\mid \msx^{(z_0)}})$ and the notion of strict specializability  along $t$ (cf. \cite[\S~3.1.1]{Mo13}). In this case we define the $\msr_\msx$-submodules $\mcm[*t]$ resp. $\mcm[!t]$ of $\mcm$, which are locally generated by $V_0^{(z_0)}\mcm$ resp. $V_{<0}^{(z_0)}\mcm$.
\begin{remark}\label{rem:naivVSastLoc}
If the coherent $\msr_\msx(*t)$-module $\mcm$ is itself $V_0 \msr_\msx$ coherent, then $\mcm[!t] = \mcm[*t] = \mcm(*t) = \mcm$.
\end{remark}

Given an $\msr_\msx(*t)$-triple $\mct = (\mcm_1,\mcm_2,C)$ which is strictly specializable along $t$ we can define
\[
\mct[!t]:= (\mcm_1[*t],\mcm_2[!t],C[!t]), \qquad \mct[*t]:=(\mcm_1[!t], \mcm_2[*t],C[*t])
\]
(cf. \cite[Prop. 3.2.1]{Mo13} for details).

The category of filtered $\msr_\msx$-triples (i.e. $\msr_\msx$-triples equipped with a finite increasing filtration $W$) underlies the category $\text{MTM}(X)$ of mixed twistor $\mcd$-modules (cf. \cite[Def. 7.2.1]{Mo13}). The full subcategory of objects $\mct \in \text{MTM}(X)$ satisfying $\mct = \mct[*D]$ for some hypersurface $D \subset X$ is denoted by $\text{MTM}(X,[*D])$.

If $X$ is a smooth, algebraic variety, we denote by $X^{\text{an}}$ the corresponding complex manifold. Let $\overline{X}$ be a smooth, complete, algebraic variety such that $X \hra \overline{X}$ is an open immersion and $D:= \overline{X} \setminus X$ is a hypersurface. We can define the category of (integrable) algebraic, mixed twistor $\mcd$-modules as
\begin{equation}\label{eq:defMTMalg}
\MTM^{\text{(int)}}_{\text{alg}}(X) := \text{MTM}^{\text{(int)}}(\overline{X}^{\text{an}},[*D]).
\end{equation}
We remark that this definition is independent of the completion up to an equivalence of categories (\hspace{-.5pt}\cite[Lem. 14.1.3]{Mo13}).

Let $f: X \ra Y$ be a quasi-projective morphism of smooth, algebraic varieties. We take completions $X \subset \overline{X}, Y \subset \overline{Y}$  as above, such that $ D_X := \overline{X} \setminus X$ and $D_Y:= \overline{Y} \setminus Y$ and we have a projective morphism $\overline{f}: \overline{X} \ra \overline{Y}$ which restricts to $f$. For $\mct \in \text{MTM}_{\text{alg}}(X)$, corresponding to $\overline{\mct} \in \text{MTM}(\overline{X},[*D_X])$, we define
\[
f_*^i \mct := \mch^i \overline{f}_* \overline{\mct}\, ,
\]
where $\overline{f}_*$ is the direct image functor for mixed twistor $\cD$-modules arising from the one for $\msr$-modules depicted in \ref{eq:ImageFunctors}.

If $X$ is an algebraic variety, we denote by $\mcd_X$ the sheaf of algebraic differential operators and by $\mcr_\msx$ the sheaf of $z$-differential operators, where here $\msx := \mba^1_z \times X$. We define the inverse and direct image functor in the category of algebraic $\mcr_\msx$-modules as in \ref{eq:ImageFunctors}. Analogously to the construction of $\cR_\msx$, we can consider the projection $p:\dP^1\times X\ra X$, and construct the sheaf of subalgebras of $\cD_{\dP^1\times X}(*(\{\infty\}\times X))$ generated by $z^2\p_z$ and $zp^*\Theta_X$ over $\cO_{\dP^1\times X}$ (cf. \cite[\S~14.4.1.1]{Mo13}), which will be denoted by $\cR_{\dP^1\times X}^{\text{int}}(*\infty)$. In that sense, an algebraic integrable $\mcr_\msx$-module gives rise to a unique $\cR_{\dP^1\times X}^{\text{int}}(*\infty)$-module (cf. [ibid., Thm. 14.4.8]).

The following Lemma, which will be needed later, is due to T. Mochizuki.

\begin{lemma}\label{lem:Moch}
Given two good $\mcr_{\mbp^1 \times X}^{\text{\emph{int}}}(*\infty)$-modules $\cP_1,\cP_2$ and an analytic isomorphism $f: \cP_1^{\text{\emph{an}}} \ra \cP_2^{\text{\emph{an}}}$, then $f$ is induced by a unique algebraic isomorphism  between $\cP_1$ and $\cP_2$.
\end{lemma}
\begin{proof}
Take a coherent $\mco_{\mbp^1 \times X}$-submodule $\cN_1 \subset \cP_1$ such that $\mcr_{\mbp^1\times X}^{\text{int}}(*\infty)\otimes \cN_1 \ra \cP_1$ is surjective and a coherent $\mco_{\mbp^1 \times X}$-module $\cN_2 \subset \cP_2$ such that both $\mcr_{\mbp^1 \times X}^{\text{int}}(*\infty) \otimes \cN_2 \ra \cP_2$ is surjective and $f(\cN_1^{\text{an}}) \subset \cN_2^{\text{an}}$. According to GAGA we have a morphism $g:\cN_1 \ra \cN_2$  which after analytification is equal to the morphism $\cN_1^{\text{an}} \ra \cN_2^{\text{an}}$ induced by $f$. Denote by $\cK_1$ the kernel of $\mcr_{\mbp^1 \times X}^{\text{int}}(*\infty) \otimes \cN_1 \ra \cP_1$. This gives a morphism $\cK_1 \ra \cP_2$ which one obtains as the composition $\cK_1 \ra \mcr_{\mbp^1 \times X}^{\text{int}}(*\infty) \otimes \cN_1 \overset{\varphi}\lra \mcr_{\mbp_1 \times }^{\text{int}}(*\infty) \otimes \cN_2 \ra \cP_2$, where $\varphi$ is induced by $g$. Because the induced morphism $(\mcr_{\mbp^1 \times X}^{\text{int}}(*\infty) \otimes \cN_1)^{\text{an}} \ra \cP_2^{\text{an}}$ factors through $\cP_1^{\text{an}}$, the induced morphism $\cK_1^{\text{an}} \ra \cP_2^{\text{an}}$ is $0$. Hence, we obtain that $\cK_1 \ra \cP_2$ is $0$, which means that $\mcr_{\mbp^1 \times X}^{\text{int}}(*\infty) \otimes \cN_1 \ra \cP_2$ factors through $\cP_1$. This shows the existence. The uniqueness follows from \cite[Prop. 10]{Serre}.
\end{proof}

Since an algebraic, integrable, mixed twistor $\mcd$-module on $X$ gives rise to an analytic $\msr_{\mbp^1 \times X}^{\text{int}}(*\infty)$-module which underlies an algebraic $\mcr_{\mbp^1 \times X}^{\text{int}}(*\infty)$-module by \cite[Thm. 14.4.8]{Mo13}, the Lemma above shows that we can define functors (up to canonical isomorphism)
\begin{align*}
\For_i: \MTM^{\text{int}}_{\text{alg}}(X) &\lra \operatorname{Mod}(\mcr_\msx^{\text{int}}) \\
(\mcm_1,\mcm_2,C) &\mapsto \mcm_i \qquad \qquad  \text{for} \; i=1,2,
\end{align*}
which become faithful if we impose goodness.

\section{Fourier transformation of twistor $\cD$-modules}

In this section we define the Fourier-Laplace transformation in the categories of integrable $\mcr$-modules and integrable, algebraic, mixed twistor $\mcd$-modules, and we prove that these two transformations are compatible.\\

Consider the diagram
\[
\xymatrix{ & \dA^N \times \widehat{\dA}^N \ar[rr]^j \ar[dl]_{p} \ar[dr]^q & & \dP^N\times \widehat{\dP}^N \ar[d]_{\overline{q}}\\ \dA^N & & \widehat{\dA}^N \ar[r]^{\widehat{j}} & \widehat{\dP}^N},
\]
where $p$ and $q$ are the projections to the first and second factor respectively.
Consider the function $\varphi=\sum_{i=1}^N w_i\cdot \lambda_i$ on $\dA^N\times \wh\dA^N$.\\

Let $\kaa$ the $\mcr_{\mba^1 \times \mba^N \times \widehat{\mba}^N}$-module $\mco_{\mba^1 \times \mba^N \times \widehat{\mba}^N}$ equipped with the $z$-connection $z d  +d\varphi$, and
consider the reduced divisor $D:= (\mbp^N \times \widehat{\mbp}^N) \setminus (\mba^N \times \widehat{\mba}^N)$. Then $\kap := j_* \kaa$ carries a natural structure of an $\mcr_{\mba^1 \times \mbp^N \times \widehat{\mbp}^N}(*D)$-module.\\

We denote by $\khn$ the analytification of $\kap$, which is an $\msr_{\mba^1 \times \mbp^N \times \widehat{\mbp}^N}(*D)$-module.

\begin{lemma}\label{lem:naivVSright}
$\khn$ is strictly specializable along $D$ and
\[
\kh :=\khn[*D] =  \khn \, .
\]
\end{lemma}
\begin{proof}
We denote the coordinates on $\mbp^N \times \widehat{\mbp}^N$ by $((w_0:w_1: \ldots :w_N),(\lambda_0:\lambda_1:\ldots:\lambda_N))$, where the chart $\mba^N \times \widehat{\mba}^N$ is  embedded via the map $j: (w_1,\ldots, w_N,\lambda_1,\ldots, \lambda_N) \mapsto ((1:w_1,\ldots,w_N),(1:\lambda_1:\ldots:\lambda_N))$. By symmetry it is enough to prove the claim in the charts $\{w_1 \neq 0,\lambda_0 \neq 0\}$, $\{w_1 \neq 0,\lambda_1 \neq 0\}$ and $\{w_1 \neq 0, \lambda_2 \neq 0\}$. We will assume $N\geq2$ and consider the chart $X:=\{w_1 \neq 0, \lambda_2 \neq 0\}$; the arguments with the other charts and when $N=1$ go similarly. The chart $X$ is embedded as $(x_1,\ldots, x_N,\mu_1,\ldots, \mu_N) \mapsto ((x_1:1:x_2\ldots:x_N),(\mu_1:\mu_2:1:\mu_3:\ldots:\mu_N))$, so that the map $\varphi$ is given on $X$ by $\frac{1}{x_1 \mu_1}(\mu_2+ x_2 + \sum_{i \geq 3}\mu_i x_i)$. Set $\msd_X := \dA^1 \times (D \cap X) = \dA^1 \times \{ x_1 \cdot \mu_1 = 0\}$. The module $(\khn)_{\mid X}$ is a cyclic $\msr_{\dA^1 \times X}(*\msd_X)$-module $\msr_{\dA^1 \times X}(*\msd_X)/ \msi$, where the left ideal $\msi$ is generated by
\begin{align*}
z\p_{x_1} + \frac{1}{x_1^2 \mu_1}(\mu_2+ x_2 + \sum_{i \geq 3} \mu_i x_i), \qquad z \p_{x_2} - \frac{1}{x_1 \mu_1}, \qquad z\p_{x_j} - \frac{\mu_j}{x_1 \mu_1}, \\
z \p_{\mu_1} + \frac{1}{x_1 \mu_1^2}(\mu_2 +x_2 + \sum_{i \geq 3} \mu_i x_i), \qquad z\p_{\mu_2} - \frac{1}{x_1\mu_1}, \qquad z \mu_j - \frac{x_j}{x_1 \mu_1},
\end{align*}
where $j \geq 3$. Consider the map $i_g: X \ra \dA^1_t \times X$ given by
\[
(x_1,\ldots,x_N,\mu_1,\ldots,\mu_N) \mapsto  (x_1\cdot \mu_1, x_1,\ldots,x_N,\mu_1,\ldots,\mu_N).
\]

The direct image $i_{g,+}(\msr_{\msx}(*\msd_X)/ \msi)$ is a cyclic $\msr_{\dA^1_t \times \msx}(*(\dA^1_t \times \msd_X))$-module $\msr_{\dA^1_t \times \msx}(*(\dA^1_t \times \msd_X)) / \msj'$ where $\msj'$ is generated by
\begin{align*}
& z\p_{x_1}+ \mu_1 z\p_t + \frac{1}{x_1^2 \mu_1}(\mu_2+x_2+ \sum_{i \geq 3} \mu_i x_i), \qquad z\p_{x_2} - \frac{1}{x_1\mu_1}, \qquad  z\p_{x_j} - \frac{\mu_j}{x_1\mu_1},\\
&z\p_{\mu_1} + x_1 z\p_t + \frac{1}{x_1 \mu_1^2}(\mu_2 + x_2 + \sum_{i \geq 3} \mu_i x_i), \qquad z\p_{\mu_2} - \frac{1}{x_1 \mu_1}, \qquad z\p_{\mu_j} - \frac{x_j}{x_1 \mu_1}, \qquad t - x_1 \mu_1,
\end{align*}
where $j \geq 3$. Define the cyclic $\msr_{\dA^1_t \times \msx}(*t)$-module $\msr_{\dA^1_t \times \msx}(*t) / \msj$ where $\msj$ is generated by
\begin{align*}
& z\p_{x_1}+ \mu_1 z\p_t + \frac{\mu_1}{t^2}(\mu_2+x_2+ \sum_{i \geq 3} \mu_i x_i), \qquad z\p_{x_2} - \frac{1}{t}, \qquad  z\p_{x_j} - \frac{\mu_j}{t},\\
&z\p_{\mu_1} + x_1 z\p_t + \frac{x_1}{t^2}(\mu_2 + x_2 + \sum_{i \geq 3} \mu_i x_i), \qquad z\p_{\mu_2} - \frac{1}{t}, \qquad z\p_{\mu_j} - \frac{x_j}{t}, \qquad t - x_1 \mu_1,
\end{align*}
where $j \geq 3$. Then we have the following $\msr_{\dA^1_t \times \msx}$-linear isomorphism
\begin{align*}
\msr_{\dA^1_t \times \msx}(*(\dA^1_t \times \msd_X)) / \msj' &\lra \msr_{\dA^1_t \times \msx}(*t) / \msj \\
P \cdot \frac{1}{(x_1 \mu_1)^k} &\mapsto P \cdot \frac{1}{t^k}.
\end{align*}
Consider the $V$-filtration along $t=0$. The relations $\frac{1}{t^k}  = (z\p_{\mu_2})^k$,
\[
z\p_t =  - \frac{1}{t} \left(z\p_{x_1} x_1 + \frac{1}{t} (\mu_2 + x_2 + \sum_{i \geq 3} \mu_i x_i) \right) =  - z\p_{x_1}x_1 z\p_{\mu_2} -  (\mu_2 + x_2 + \sum_{i \geq 3} \mu_i x_i)(z\p_{\mu_2})^2
\]
and a straightforward induction over $k$ for $(z \p_t)^k$  show that $i_{g,+} (\msr_{\msx}(*\msd_X) /\msj)$ is a cyclic, hence also coherent, $V_0 \msr_{\dA^1_t \times \msx}$-module. It follows from Remark \ref{rem:CoherSpecAlongT} that $i_{g,+} (\msr_{\msx}(*\msd_X) /\msj) = i_{g,+} (\msr_{\msx}(*\msd_X) /\msj)[*t]$, and as a consequence, we are done by applying \cite[\S~3.3.1.1]{Mo13} and Remark \ref{rem:naivVSastLoc}.

\end{proof}

It follows from \cite[Prop. 3.3]{SabYu} that $\kh$ underlies an object $\kmtm \in \text{MTM}^{\text{int}}_{\text{alg}}(\mba^N \times \widehat{\mba}^N)$.

We will now define a Fourier-Laplace transformation for algebraic $\mcr^{\text{int}}_{\mba^1 \times \mba^N}$-modules.

\begin{defn}
The Fourier-Laplace transformation functor from the category of algebraic $\mcr^{\text{int}}_{\mba^1 \times \mba^N}$-modules
to the category of algebraic $\mcr^{\text{int}}_{\mba^1 \times \widehat{\mba}^N}$-modules is defined as
\[
\widehat{\mcm}:=\FL (\mcm) := \mch^0 q_+ \left((p^+ \mcm) \otimes \kaa \right),
\]
for any $\mcm$ in $\text{Mod}(\mcr^{\text{int}}_{\mba^1 \times \mba^N})$.
\end{defn}

\begin{remark}\label{rem:naivFL}
Let $M := \Gamma(\mba^N,\mcm)$ be the $R^{\text{int}}_{\mba^1 \times \mba^N}$-module of global sections of $\mcm$. The $R^{\text{int}}_{\mba^1 \times \widehat{\mba}^N}$-module $\widehat{M} := \Gamma(\widehat{\mba}^N,\widehat{\mcm})$ is isomorphic to $M$ as a $\mbc[z]$-module and the full $R^{\text{int}}_{\mba^1 \times \widehat{\mba}^N}$-structure is given by
\[
\lambda_i \cdot m  := -z\p_{w_i} \cdot m, \qquad z\p_{\lambda_i} \cdot m  := w_i \cdot m \qquad \text{and} \qquad  z^2\p_z \cdot m := \left(z^2 \p_z - \sum_{i=1}^N z \p_{w_i} w_i\right) \cdot m\, .
\]
On the other hand, there is a similar definiton of a Fourier-Laplace transformation in the category of algebraic $\mcd_{\mba^N}$-modules (see e.g. \cite[Def. 1.2]{Reich2}) which we also denote by $\FL$.
\end{remark}

The Fourier-Laplace transformation for algebraic, integrable, mixed twistor $\mcd$-modules is defined in the following way.

\begin{defn}
The Fourier-Laplace transformation in the category of algebraic, integrable mixed twistor $\cD$-modules on $\dA^N$ is defined by
\[
\FL_{\textup{MTM}}(\mcm) := \mch^0 q_* \left( (p^* \mcm) \otimes \kmtm \right),
\]
where $\mcm \in \textup{MTM}^{\text{int}}_{\text{alg}}(\mba^N)$.
\end{defn}

Recall that for $\mcm = (\mcm_1,\mcm_2,C) \in \textup{MTM}_{\text{alg}}^{\text{int}} (X)$ we denote by $\For_i$ the forgetful functors $\For_i(\mcm) = \mcm_i$ for $i=1,2$.
\begin{proposition}\label{prop:FLnaivVSmtm}
Let $\cM\in\MTM^{\text{\emph{int}}}_{\text{\emph{alg}}}(\mba^N)$. Then
\[
\textup{For}_1( \FL_{\textup{MTM}}(\mcm)) = \FL (\textup{For}_1(\mcm)) \qquad \text{and} \qquad \textup{For}_2( \FL_{\textup{MTM}}(\mcm)) = z^{-N} \FL (\textup{For}_2(\mcm)).
\]
\end{proposition}
\begin{proof}
By  \cite[\S~14.3.3.3]{Mo13} it is clear that $\text{For}_i$ almost commutes  with $p^*$, more precisely we have
\[
\For_1(p^*(\mcm)) = z^N p^+ (\For_1(\mcm)) \qquad \text{and} \qquad \For_2(p^*(\mcm)) = p^+ (\For_2(\mcm)).
\]
Then it is enough to prove for $\mcn \in \MTM_{\text{alg}}^{\text{int}}(\mba^N \times \widehat{\mba}^N)$ that $q_+(\For_i(\mcn) \otimes \kaa) \cong \For_i(q_*(\mcn \otimes \kmtm))$.
We have
\begin{align*}
\widehat{j}_+ q_+\left(\For_i(\mcn) \otimes \kaa\right) &\cong \overline{q}_+ j_+ \left(\For_i(\mcn) \otimes \kaa\right) \cong \overline{q}_+ j_* \left(\For_i(\mcn) \otimes \kaa\right) \\
&\cong \bR\overline{q}_* \DR_{\mbp^N \times \widehat{\mbp}^N} j_*\left(\For_i(\mcn) \otimes \kaa\right).
\end{align*}
Since $\mcn, \kmtm \in \MTM^{\text{int}}_{\text{alg}}(\mba^N \times \widehat{\mba}^N)$, there exist mixed twistor $\cD$-modules $\overline{\mcn}, \overline{\mct}^{\varphi/z} \in \MTM^{\text{int}}(\mbp^N \times \widehat{\mbp}^N,[*D])$ whose underlying $\msr$-modules are (after stupid localization along $D$) analytifications of the $j_*\For_i(\mcn)$ and $j_*\kaa$. Hence
\[
\left(j_*\left(\For_i(\mcn) \otimes \kaa\right)\right)^{\text{an}} \cong \For_i\left(\overline{\mcn} \otimes \overline{\mct}^{\varphi/z}\right)(*D) \cong \For_i\left(\overline{\mcn} \otimes \overline{\mct}^{\varphi/z}\right),
\]
where the last equation follows from Lemma \ref{lem:naivVSright}. We therefore get
\begin{align*}
\left( \widehat{j}_+ p_+\left(\For_i(\mcn) \otimes \kaa\right) \right)^{\text{an}} &\cong \bR\overline{q}_* \DR^{\text{an}}_{\mbp^N \times \widehat{\mbp}^N} \left( j_*\left(\For_i(\mcn) \otimes \kaa\right)\right)^{\text{an}} \\
&\cong \bR\overline{q}_* \DR^{\text{an}}_{\mbp^N \times \widehat{\mbp}^N} \For_i\left(\overline{\mcn} \otimes \overline{\mct}^{\varphi/z}\right) \\
&\cong \For_i\left( \overline{q}_* \left(\overline{\mcn} \otimes \overline{\mct}^{\varphi/z}\right) \right).
\end{align*}
The claim follows now from Lemma \ref{lem:Moch}, noting that the goodness is a consequence of Lemma \ref{lem:naivVSright} and \cite[Thm. 14.4.15]{Mo13}.
\end{proof}

We have the following variant, which will be used in the next section. Consider the diagram
$$
\xymatrix{ & \dA^N \times \dG_m \ar[rr]^j \ar[dl]_{p} \ar[dr]^q & & \dP^N\times \dP^1 \ar[d]_{\overline{q}}\\ \dA^N & & \dG_m \ar[r]^{\widehat{j}} & \mbp^1}
$$
and let $\psi:=w_1\cdot t+w_2+\ldots+w_N$.

Similarly as above we define the $\mcr_{\mba^1 \times \mba^N \times \dG_m}$-module $\laa$, being $\mco_{\mba^1 \times \mba^N \times \dG_m}$ endowed with the $z$-connection $z d + d\psi$. As in the other case, we can consider the divisor $H:=(\mbp^N \times \mbp^1) \setminus (\mba^N \times \mbg_m)$ and obtain the $\mcr_{\mba^1 \times \mbp^N \times \mbp^1}(*H)$-module $\lap:=j_*\laa$. In the same vein as before, we will denote by $\lhn$ the $\msr_{\mba^1 \times \mbp^N \times \mbp^1}(*H)$-module being the analytification of $\lap$. The following Lemma is similar to Lemma \ref{lem:naivVSright}.

\begin{lemma}\label{lem:adaptedVSright}
$\lhn$ is strictly specializable along $H$ and
\[
\lh := \lhn[*H]=\lhn.
\]
\end{lemma}
\begin{proof}
We denote the coordinates on $\mbp^N \times \mbp^1$ by $((w_0:w_1:\ldots:w_n),(u:t))$, where the chart $\mba^N \times \mbg_m$ is embedded via the map $j: (w_1,\ldots,w_N,t) \mapsto((1:w_1:\ldots:w_N),(1:t))$. We will assume $N\geq3$ and consider the chart $X:=\{w_2 \neq 0,u \neq 0\}$; the other charts behave similarly, as well as the case $N=1,2$. The chart $X$ is embedded as $(x_1,\ldots,x_N,u) \mapsto ((x_1:x_2:1:x_3:\ldots:x_N),(u:1))$. On this chart the map $\psi$ is given by $\frac{1}{x_1}(\frac{x_2}{u} +1 + x_3+ \ldots+ x_N)$. Set $\msh_X:= \dA^1_s \times ( H \cap X) = \dA^1_s \times \{x_1\cdot   u = 0\}$. The module $(\lhn)_{\mid X}$ is a cyclic $\msr_{\msx}(*\msh_X)$-module $\msr_\msx(*\msh_X)/ \msi$, where the left ideal $\msi$ is generated by
$$
z\p_{x_1} + \frac{1}{x_1^2}\left(\frac{x_2}{u}+1+x_3+\ldots+x_N \right), \qquad z \p_{x_2}-\frac{1}{x_1u}, \qquad z \p_{x_j} - \frac{1}{x_1}, \qquad z \p_u+ \frac{x_2}{x_1 u^2}, $$
with $j \geq 3$. Consider the map $i_g: X \ra \dA^1_s \times X$ given by
\[
(x_1,\ldots,x_N,u) \mapsto (x_1 \cdot u, x_1, \ldots, x_N,u)\, .
\]
Analogously as in Lemma \ref{lem:naivVSright}, the direct image $i_{g,+} (\msr_\msx(*H_X) / \msj)$  is a cyclic $\msr_{\dA^1_s \times \msx}(*(\dA^1_s\times \msh_X))$-module $\msr_{\dA^1_s \times \msx}(*(\dA^1_s\times \msh_X))/\msj'$ where $\msj'$ is the left ideal generated by
\begin{align*}
&z\p_{x_1} + u z\p_s + \frac{1}{x_1^2}\left(\frac{x_2}{u}+1+x_3+\ldots+x_N \right), \qquad z \p_{x_2}-\frac{1}{x_1u}, \qquad z \p_{x_j} - \frac{1}{x_1}, \\
&z \p_u+ x_1z\p_s + \frac{x_2}{x_1 u^2}, \qquad s - x_1 u,
\end{align*}
and $j \geq 3$. Define the cyclic $\msr_{\dA^1_s \times \msx}(*s)$-module $\msr_{\dA^1_s \times \msx}(*s) /\msj$ where $\msj$ is generated by
\begin{align*}
&z\p_{x_1} + u z\p_s + \frac{1}{s^2}\left(x_2 u+u^2+x_3u^2+\ldots+x_Nu^2 \right), \qquad z \p_{x_2}-\frac{1}{s}, \qquad z \p_{x_j} - \frac{u}{s}, \\
&z \p_u+ x_1z\p_s + \frac{x_1x_2}{s^2}, \qquad s - x_1 u,
\end{align*}
where $j \geq 3$. We have the following $\msr_{\dA^1_s \times \msx}$-linear isomorphism
\begin{align*}
\msr_{\dA^1_s \times \msx}(*(\dA^1_s\times \msh_X))/\msj' &\lra \msr_{\dA^1_s \times \msx}(*s) /\msj \\
P \frac{1}{(x_1 u)^k} &\mapsto P \frac{1}{s^k}.
\end{align*}

Consider the $V$-filtration along $s=0$. The relations $\frac{1}{s^k}  = (z \p_{x_2})^k$,
\[
z\p_s =  - \frac{1}{s}\left( z + u z \p_u + \frac{x_2}{s}\right)=-z \cdot z\p_{x_2} - u z \p_u z \p_{x_2} - x_2 (z \p_{x_2})^2
\]
and a straightforward induction over $k$ for $(z \p_s)^k$ show that $i_{g,+} (\msr_\msx(*D_X) /\msj)$ is a coherent $V_0 \msr_{\dA^1_t \times \msx}$-module. As in the previous lemma, this shows the claim.

\end{proof}

It follows again from \cite[Prop. 3.3]{SabYu} that $\lh$ underlies an object $\lmtm \in \text{MTM}^{\text{int}}_{\text{alg}}(\mba^n \times \dG_m)$.

\begin{defn}$ $\\[-1em]
\begin{enumerate}
\item The Fourier-Laplace transformation with respect to the kernel $\psi$  in the category of algebraic $\mcr_{\mba^1 \times \mba^N}$-modules is defined as
\[
 \FL^\psi (\mcm) := \mch^0 q_+ \left((p^+ \mcm) \otimes \laa \right),
\]
for any $\mcm \in \text{Mod}(\mcr_{\mba^N})$.
\item Analogously, the Fourier-Laplace transformation with respect to the kernel $\psi$ in the category of algebraic, integrable twistor $\cD$-modules on $\dA^N$ is defined by
\[
\FL^\psi_{\textup{MTM}}(\mcm) := \mch^0 q_* \left( (p^* \mcm) \otimes \lmtm \right),
\]
for any $\mcm \in \textup{MTM}^{\text{int}}_{\text{alg}}(\mba^N)$.
\end{enumerate}
\end{defn}

We get the following result for the kernel $\psi$.

\begin{proposition}\label{prop:1dimFL}
Let $\cM\in\MTM^{\text{\emph{int}}}_{\text{\emph{alg}}}(\mba^N)$. Then
\[
\textup{For}_1( \FL^\psi_{\textup{MTM}}(\mcm)) =z^{1-N} \FL^\psi (\textup{For}_1(\mcm)) \qquad \text{and} \qquad \textup{For}_2( \FL^\psi_{\textup{MTM}}(\mcm)) = z^{-N}\FL^\psi (\textup{For}_2(\mcm))\, .
\]
\end{proposition}
\begin{proof}
We have, by \cite[\S~14.3.3.3]{Mo13},
\[
\For_1(p^*(\mcm)) = z p^+ (\For_1(\mcm)) \qquad \text{and} \qquad \For_2(p^*(\mcm)) = p^+ (\For_2(\mcm)) .
\]
The rest of the proof carries over almost word for word from Proposition \ref{prop:FLnaivVSmtm}, using Lemma \ref{lem:adaptedVSright}.
\end{proof}

\section{GKZ systems and irregular Hodge modules}\label{sec:TorusHodge}

Let $A=(a_{ki})$ be a $d \times N$ integer matrix with columns $(\underline{a}_1,\ldots, \underline{a}_N)$. We define
\[
\mbn A := \sum_{i=1}^N \mbn \underline{a}_i \subset \mbz^d
\]
and similarly for $\mbz A$ and $\mbr_{\geq 0} A$. Throughout this section we assume
\[
\mbz A = \mbz^d \qquad \text{and} \qquad \mbn A = \mbz^d \cap \mbr_{\geq 0} A\, .
\]
Set $\mba^N := \Spec (\mbc[w_1,\ldots,w_N])$ and $\widehat{\mba}^N := \Spec (\mbc[\lambda_1,\ldots,\lambda_N] )$ and
$$\mbl_A := \left\{ \ell = (\ell_1,\ldots, \ell_N) \in \mbz^N \,:\, \sum_{i=1}^N \ell_i \underline{a}_i\right\}.$$
\begin{definition}
The GKZ-hypergeometric system $\mcm^\beta_A$ is the cyclic $\mcd_{\widehat{\mba}^N}$-module $\mcd_{\widehat{\mba}^N} / \mci$, where $\mci$ is the left ideal generated by
\[
E_k := \sum_{i=1}^N a_{ki} \lambda_i \p_{\lambda_i} -\beta_k\,, \text{ for }k=1,\ldots,d,
\]
and
\[
\Box_\ell :=\prod_{\ell_i > 0} \p_{\lambda_i}^{\ell_i} - \prod_{\ell_i < 0} \p_{\lambda_i}^{- \ell_i}\,, \text{ for }l\in\dL_A.
\]
\end{definition}

The GKZ-hypergeometric system $\mcm^\beta_A$ is the Fourier-Laplace transform of the cyclic $\mcd_{\mba^N}$-module $\check{\mcm}^\beta_A := \mcd_{\mba^N} / \mcj$, where $\mcj$ is the left ideal generated by
\[
\check{E}_k := \sum_{i=1}^N a_{ki} \p_{w_i} w_i + \beta_k\,, \text{ for }k=1,\ldots,d,
\]
and
\[
\check{\Box}_\ell :=\prod_{\ell_i > 0} w_i^{\ell_i} - \prod_{\ell_i < 0} w_i^{- \ell_i}\,, \text{ for }l\in\dL_A.
\]

The semigroup ring $\mbc[\mbn A] \subset \mbc[t_1^\pm,\ldots, t_d^\pm]$ is naturally a $\mbc[w_1,\ldots, w_N]$-module under the isomorphism
\begin{align*}
\mbc[w_1,\ldots, w_N]/ ((\check{\Box}_\ell)_{\ell \in \mbl_A}) &\lra \mbc[\mbn A] \\
w_i & \longmapsto t^{\underline{a}_i},
\end{align*}
where we are using the multi-index notation $t^{\underline{a}_i} := \prod_{k=1}^d t_k^{a_{ki}}$. We set $S_A := \mbc[\mbn A]$. Notice that the rings $\mbc[w_1,\ldots, w_N]$ and $S_A$ carry a natural $\mbz^d$-grading given by $\deg(w_i) = \underline{a}_i$. This is compatible with the grading on the Weyl algebra $D_{\mba^N}:= \Gamma(\mba^N, \mcd_{\mba^N})$ given by $\deg(w_i) = \underline{a}_i$ and $\deg(\p_{w_i}) = - \underline{a}_i$.

\begin{definition}(\hspace{-.5pt}\cite[Def. 5.2]{MillerWaltherMat})
Let $P$ be a finitely generated $\mbz^d$-graded $\mbc[w_1,\ldots,w_N]$-module. An element $\alpha \in \mbz^d$ is called a true degree of $P$ if the graded part $P_\alpha$ is non-zero. A vector $\alpha \in \mbc^d$ is called a quasi-degree of $P$ if $\alpha$ lies in the complex Zariski closure $qdeg(P)$ of the true degrees of $P$ via the natural embedding $\mbz^d \hookrightarrow \mbc^d$.
\end{definition}
Consider the set of strongly resonant parameters of $A$:
\[
sRes(A) := \bigcup_{j=1}^N sRes_j(A),
\]
where
\[
sRes_j(A) := \{\beta \in \mbc^d \mid \beta \in -(\mbn+1)\underline{a}_j + qdeg(S_A/(t^{\underline{a}_j}))\}.
\]
Consider as well the torus $\mbg^d_m := \Spec(\mbc[t_1^\pm,\ldots,t_d^\pm])$, together with the torus embedding
\begin{align*}
h: \mbg^d_m &\longrightarrow \mba^N \\
(t_1,\ldots, t_d) &\mapsto (t^{\underline{a}_1},\ldots, t^{\underline{a}_N}).
\end{align*}
The following proposition is a slight generalization of the results of Schulze and Walther \cite[Thm. 3.6, Cor. 3.8]{SchulWalth2}.

\begin{proposition}\emph{(\hspace{-.5pt}\cite[Prop. 2.11]{RS15})}\label{prop:DirectImage}
Let $A$ be a $d\times N$ integer matrix satisfying $\mbz A =\mbz^d$ and $\mbn A = \mbz^d \cap \mbr_{\geq 0}A$. Assume that $\beta \not \in sRes(A)$. Then
\[
\mch^0\left(h_+ \mco_{\mbg^d_m}^\beta \right) \cong \check{\mcm}^\beta_A,
\]
where $\mco_{\mbg^d_m}^\beta \cong \mcd_{\mbg^d_m} /\mcd_{\mbg^d_m} \cdot(\p_{t_1}t_1+\beta_1,\ldots,\p_{t_d} t_d + \beta_d)$
\end{proposition}

For $\beta \in \mbr^d$, the $\mcd$-module $\mco^\beta_{\mbg^d_m}$ underlies the complex mixed Hodge module ${^p}\mbc^{H,\beta}_{\mbg^d_m}$. Hence for $\beta \in \mbr^d \setminus sRes(A)$ the $\mcd$-module $\check{\mcm}^\beta_A$ underlies the complex mixed Hodge module $\mch^0 h_* {^p}\mbc^{H,\beta}_{\mbg^d_m}$. The Hodge filtration on $\check{\mcm}^\beta_A$ can be explicitly computed, provided that $\beta$ belongs to a certain set $\mathfrak{A}_A$ of so-called admissible parameters $\beta$. We recall its definition from \cite[p. 11]{RS15}: Let $\underline{c}:= \underline{a}_1+\ldots+\underline{a}_N$ and define for all facets $F$ of $\mbr_{\geq 0}A$ the uniquely determined primitive, inward-pointing, normal vector $\underline{n}_F$ of $F$, such that $\langle \underline{n}_F,F\rangle = 0$ and $\langle \underline{n}_F,\mbn A \rangle \subset \mbz_{\geq 0}$. Set $e_F:= \langle \underline{n}_F,\underline{c}\rangle \in \mbz_{>0}$. The set of admissible parameters  of $A$ is then defined by
\[
\mathfrak{A}_A:= \bigcap_{F \text{ facet}} \{ \mbr \cdot F - [0, 1/e_F) \cdot \underline{c}\}\,.
\]

\begin{theorem}\emph{(\hspace{-.5pt}\cite[Thm. 3.16]{RS15})}\label{thm:ClosedImmersion}
For $\beta \in \mathfrak{A}_A$ the Hodge filtration on $\check{\mcm}^\beta_A$ is equal to the order filtration shifted by $N-d$, i.e.
\[
F^H_{p+{N-d}}\check{\mcm}^\beta_A = F_p^{\text{\emph{ord}}} \check{\mcm}^\beta_A\, .
\]
\end{theorem}

Let us define the cyclic $\mcr_{\mba^1 \times \mba^N}$-module $\check{\mcn}^\beta_A :=\mcr_{\mba^1 \times \mba^N}/ \mcj_z$, where $\mcj_z$ is the left ideal generated by
\[
\check{E}^z_k = \sum_{i=1}^N a_{ki} z\p_{w_i} w_i + z\beta_k\,, \text{ for }k=1,\ldots,d,
\]
and
\[
\check{\Box}_\ell =\prod_{\ell_i > 0} w_i^{\ell_i} - \prod_{\ell_i < 0} w_i^{- \ell_i}\,, \text{ for }l\in\dL_A.
\]
We will denote by $\check{M}^\beta_A := \Gamma(\mba^N, \check{\cM}^\beta_A)$ and $\check{N}^\beta_A := \Gamma(\mba^1 \times \mba^N, \check{\mcn}^\beta_A)$ the modules of global sections of $\check{\cM}^\beta_A$ and $\check{\mcn}^\beta_A$, respectively.

We will also consider the Rees module of $\check{M}^\beta_A$ with respect to the order filtration $F^{\text{ord}}_\bullet$, which is given by $R^{F^{\text{ord}}} \check{M}^\beta_A := \sum_{k \geq 0} z^k F^{\text{ord}}_k \check{M}^\beta_A$. An easy computation shows
$R^{F^{\text{ord}}} \check{M}^\beta_A = \check{N}^\beta_A$, hence
\begin{equation}\label{eq:ReesSW}
R^{F^H} \check{M}^\beta_A = z^{N-d} \check{N}^\beta_A.
\end{equation}

\begin{definition}
The $\mcr$-GKZ-hypergeometric system $\mcn^\beta_A$ is the cyclic $\mcr_{\mba^1 \times \widehat{\mba}^N}^{\text{int}}$-module $\mcr_{\mba^1 \times \widehat{\mba}^N}^{\text{int}} / \mci$, where the left ideal $\mci$ is generated by
\begin{align*}
&E^z_0 := z^2\p_z + \sum_{i=1}^N \lambda_i z\p_{\lambda_i},\\
&E^z_k := \sum_{i=1}^N a_{ki} \lambda_i z\p_{\lambda_i} -z\beta_k, \text{ for } k=1,\ldots,d,
\end{align*}
and
\[
\Box^z_\ell :=\prod_{\ell_i > 0} (z\p_{\lambda_i})^{\ell_i} - \prod_{\ell_i < 0} (z\p_{\lambda_i})^{- \ell_i}, \text{ for } \ell \in \mbl_A\, .
\]
\end{definition}

\begin{remark}
Note that, considering $\check\cN_A^\beta$ as an $\cR_{\dA^1\times\wh\dA^N}^{\text{int}}$-module with the trivial action of $z^2\dd_z$, $\cN_A^\beta$ is its Fourier-Laplace transform as $\cR_{\dA^1\times\wh\dA^N}^{\text{int}}$-modules, according to Remark \ref{rem:naivFL}.
\end{remark}

\begin{theorem}
Let $A$ be a $d\times N$-matrix and $\beta \in \mathfrak{A}_A$ an admissible parameter.  The $\mcr$-GKZ-hypergeometric system $z^{-d} \mcn^\beta_A$ underlies an algebraic, integrable, mixed twistor $\mcd$-module ${^\mct\!\!}\mcm^\beta_A$.
\end{theorem}
\begin{proof}
By the Remark above, we know that $\mcn^\beta_A=\FL(\check\cN_A^\beta)$, which in turn, thanks to the choice of $\beta$, Theorem \ref{thm:ClosedImmersion} and formula \eqref{eq:ReesSW}, is equal to $\FL(z^{d-N}\cR^{F^H}\check\cM_A^\beta)$. Since $\cR^{F^H}\check\cM_A^\beta$ is the Rees module of a mixed Hodge module on $\dA^N$, it gives rise to an algebraic, integrable mixed twistor $\cD$-module on $\dA^N$, say ${^\cT\!\!}\check\cM_A^\beta$. Then we can apply Proposition \ref{prop:FLnaivVSmtm} and get
$$\mcn^\beta_A=z^{d-N}\FL\left(\For_2\left({^\cT\!\!}\check\cM_A^\beta\right)\right)= z^d\For_2\left(\FL_{\MTM}\left({^\cT\!\!}\check\cM_A^\beta\right)\right).$$
The result follows from writing ${^\mct\!\!}\mcm^\beta_A:=\FL_{\MTM}\left({^\cT\!\!}\check\cM_A^\beta\right)$.
\end{proof}

\begin{corollary}
The analytification of ${^\mct\!\!}\mcm^\beta_A$ gives rise to an irregular mixed Hodge module on $\dA^N$ which has a natural extension to an $\mcr^{\text{\emph{int}}}_{\dA^1 \times \mbp^N}$-module underlying an object of $\IrrMHM(\mbp^N)$.
\end{corollary}
\begin{proof}
This follows from applying \cite[Cor. 0.5]{Sa15} to the operations performed to get ${^\mct\!\!}\mcm^\beta_A$.
\end{proof}

\section{Application to confluent hypergeometric systems}

In this section we are going to use the results achieved so far for
the special case of the matrix
$$A=\left(\begin{array}{c|c|c}
\ul{1}_m & \ul{0}_{m\times(n-1)} & \operatorname{Id}_m\\[3pt]
\hline & & \vspace{-10pt}\\
\ul{1}_{n-1} & -\operatorname{Id}_{n-1} & \ul{0}_{(n-1)\times m}\end{array}\right).$$

For the sake of simplicity, we will write $N=n+m$ in the following. Before going on, let us introduce the main object of study of this section and state some of its basic properties, extending what we mentioned in the introduction.

\begin{defi}\label{def:ClassicHyp}
Let $(n,m)\neq(0,0)$ be a pair of nonnegative integers, and let $\alpha_1,\ldots,\alpha_n$ and $\beta_1,\ldots,\beta_m$ be elements of $\dC$. The hypergeometric $\cD$-module of type $(n,m)$ associated with the $\alpha_i$ and the $\beta_j$ is defined as the quotient of $\cD_{\dG_{m}}$ by the left ideal generated by the so-called hypergeometric operator
$$\prod_{i=1}^n(t\dd_t-\alpha_i)-t\prod_{j=1}^m(t\dd_t-\beta_j).$$
We will denote it by $\cH(\alpha_i;\beta_j)$.
\end{defi}

\begin{prop}\label{prop:factsClassicHyp}
Let $\cH:=\cH(\alpha_i;\beta_j)$ be a hypergeometric $\cD$-module of type $(n,m)$, and let $\eta$ be any complex number. Then we have the following:
\begin{enumerate}
\item If we denote the Kummer $\cD$-module $\cD_{\dG_{m}}/(t\dd_t-\eta)$ by
$\cK_\eta$, then $\cH\otimes_{\cO_{\dG_{m}}}\cK_\eta\cong \cH(\alpha_i+\eta;\beta_j+\eta)$. In particular, an overall integer shift of the parameters gives us an isomorphic $\cD$-module.
\item $\cH$ is irreducible if and only if for any pair $(i,j)$ of indices, $\alpha_i-\beta_j$ is not an integer.
\item If $\cH$ is irreducible, its isomorphism class depends only on the classes modulo $\dZ$ of the $\alpha_i$ and the $\beta_j$, so we can choose such parameters on a fundamental domain of $\dC/\dZ$.
\end{enumerate}
\end{prop}
\begin{proof}
A simple calculation shows point 1. Point 2 follows from \cite[Prop. 2.11.9, 3.2]{Ka}, whereas point 3 is part of [ibid., Prop. 3.2].
\end{proof}

As we mentioned in the introduction, we can express any one-dimensional hypergeometric $\cD$-module as the inverse image of a GKZ hypergeometric $\cD$-module (cf. \cite[Cor. 2.9]{CDS}). Notice
that there is a similar statement at the level of $\cR$-modules (see [ibid., Lem. 2.12]), yielding
a description of the $\cRint_{\dA^1_z\times \dG_{m,t}}$-module $\wh\cH$ from Theorem \ref{thm:1dimhypgeo} below
as an inverse image of a GKZ-hypergeometric $\cR$-module (as defined in [ibid., Def. 2.10]).

\begin{proposition}\label{prop:hypGKZ}
Let $\cH(\alpha_i;\beta_j)$ be a hypergeometric $\cD_{\dG_m}$-module of type $(n,m)$ with $\alpha_1=0$, let $A\in\operatorname{M}((N-1)\times N,\dZ)$ as right above, and let $\gamma=(\beta_1,\ldots,\beta_m,\alpha_2,\ldots,\alpha_n)^{\operatorname{t}}$. Let $\iota:\dG_m\ra\dA^N$ be given by $t\mapsto(t,1\ldots,1)$. Then
$$\cH(\alpha_i;\beta_j)\cong\iota^+\cM_A^\gamma.$$
\end{proposition}

Since the restriction map $\iota$ is not smooth we do not know a priori whether taking inverse image by it
preserves irregular mixed Hodge modules. In order to show that $\cH(\alpha_i;\beta_j)$ can be upgraded to an element of $\IrrMHM(\dG_m)$ we use Proposition \ref{prop:1dimFL}, where the reduction procedure is build in by the use of the Fourier kernel $\psi=w_1\cdot t + w_2 + \ldots + w_N$.

Let $A\in \textup{M}((N-1)\times N,\dZ)$ as above and $\gamma=(\gamma_1,\ldots,\gamma_{N-1})^{\operatorname{t}}\in\fA_A$.
The $\cD_{\dA^N}$-module $\check{\mcm}^\gamma_A$ underlies a mixed Hodge module on $\dA^N$, so that the Rees module
$\cR^{F^H}\left(\check{\mcm}^\gamma_A\right)$ then gives rise to an algebraic, integrable mixed twistor $\cD$-module on $\dA^N$
that we denote by ${^\cT\!\!}\check\cM_A^\gamma$. Then we have the following concrete description of its Fourier-Laplace transform
$\FL_{\MTM}^\psi\left({^\cT\!\!}\check\cM_A^\gamma\right)
=q_*\left(p^*\left({^\cT\!\!}\check\cM_A^\gamma\right)\otimes\cT^{\psi/z}\right)$.
\begin{proposition}\label{prop:naiveFT}
Let $A$ and $\gamma$ be as before. Then the $\cRint_{\dA^1\times\dG_m}$-module $\For_2\left(\FL_{\MTM}^\psi\left({^\cT\!\!}\check\cM_A^\gamma\right)\right)$ can be expressed as $\cRint_{\dA^1\times\dG_m}/(P,H)$, where
$$
P=z^2\partial_z+(n-m)tz\partial_t+\varepsilon z \,\text{ and }\,H=zt\dd_t\prod_{i=1}^{n-1} z(t\partial_t-\gamma_{m+i})-t\prod_{j=1}^m z(t\partial_t-\gamma_j),
$$
with $\varepsilon =\sum_{j=1}^m\gamma_j-\sum_{i=m+1}^{N-1}\gamma_i+N-1$.
\end{proposition}
\begin{proof}
As said after Theorem \ref{thm:ClosedImmersion}, for any $\gamma$ inside the domain $\fA_A$ of admissible parameters, the Hodge filtration of $\check{\cM}^\gamma_A$ is the order filtration shifted by $N-(N-1) =1$. Therefore, for such values of $\gamma$ we can give an explicit expression of the Rees module of the filtered module $\left(\check{\mcm}^\gamma_A, F_\bullet^H\right)$. Namely, we have the isomorphism of $\cRint_{\dA^1\times\dA^N}$-modules
$$
\cR^{F^H}\left(\check{\mcm}^\gamma_A\right)\cong z\check{\cN}^\gamma_A:=\cRint_{\dA^1\times\dA^N}/\left(\check E^z_i,\check E^z_j,\check\Box, z^2\dd_z-z\right),
$$
where
\begin{align*}
\check E^z_i&=z\dd_{w_1}w_1-z\dd_{w_i}w_i+\gamma_{m+i-1}z,\text{ for }i=2,\ldots,n,\\
\check E^z_j&=z\dd_{w_1}w_1+z\dd_{w_{n+j}}w_{n+j}+\gamma_jz,\text{ for }j=1,\ldots,m,\\
\check\Box&=\D\prod_{i=1}^nw_i-\prod_{j=1}^mw_{n+j}.
\end{align*}

First we compute $\FL^\psi(z\check{\cN}^\gamma_A)$, which involves performing three operations with $z\check{\cN}^\gamma_A$: inverse image by $p:\dG_m\times\dA^N\ra\dA^N$, tensor product with the $\cRint_{\dA^1\times\dG_m\times\dA^N}$-module $\laa$ and direct image by $q:\dG_m\times\dA^N\ra\dG_m$. The first one is pretty easy. Namely
$$p^+ z\check{\mcn}^\gamma_A\cong\cRint_{\dA^1\times\dG_m\times\dA^N}/(\check E^z_i,\check E^z_j,\check\Box, z^2\dd_z-z,z\dd_t).$$
Let us tensor now $p^+ z\check{\mcn}^\gamma_A$ with $\laa$. This $\mcr^{\text{int}}$-module can be presented as $\cRint_{\dA^1\times\dG_m\times\dA^N}\cdot e^{\psi/z}=\cRint_{\dA^1\times\dG_m\times\dA^N}/\cI^\psi$, where $\cI^\psi$ is the left ideal generated by
$$z^2\dd_z+w_1t+w_2+\ldots+w_N,\quad z\dd_t-w_1,\quad z\dd_{w_1}-t,\quad z\dd_{w_i}-1,i=2,\ldots,N.$$
For $n \in p^+z\check{\mcn}^\gamma_A$, we will call $n^\psi$ the tensor $n\otimes e^{\psi/z}$. Then we can obtain the formulas
\begin{align*}
\left( z\p_{w_1} w_1 n \otimes e^{\psi/z}\right) &= z\p_{w_1}\left(w_1n \otimes e^{\psi/z}\right) - t \left(n \otimes w_1e^{\psi/z}\right) = (z\p_{w_1} w _1 - t z\p_t)\cdot n^\psi, \\
\left( z\p_{w_k} w_k n \otimes e^{\psi/z}\right) &= z\p_{w_k}\left(w_kn \otimes e^{\psi/z}\right) -  \left(n \otimes w_ke^{\psi/z}\right) = (z\p_{w_k} w _k - w_k)\cdot n^\psi \text{ for } k=2,\ldots ,N,\\
\left(z^2 \p_z n \otimes e^{\psi/z}\right) &= z^2 \p_z\cdot n^\psi - \left( n \otimes (-\psi)e^{\psi/z}\right) = \left( z^2 \p_z +w_1 t + w_2 + \ldots + w_N \right)\cdot n^\psi,\\
\left(z \p_t n \otimes e^{\psi/z} \right) &= z\p_t\cdot n^\psi - \left(n \otimes w_1 e^{\psi/z}\right) = (z \p_t -w_1)\cdot n^\psi.
\end{align*}

Hence $p^+z\check{\mcn}^\gamma_A \otimes \laa$ is the cyclic $\mcr^{\text{int}}_{\mba^1 \times \mbg_m \times \mba^n }$-module $\mcr^{\text{int}}_{\mba^1 \times \mbg_m \times \mba^n }/\cJ^\psi$, with $\cJ^\psi$ being the left ideal generated by
\begin{align*}
&\D\prod_{i=1}^nw_i-\prod_{j=1}^mw_{n+j}, \quad z^2 \p_z -z+w_1t+w_2 + \ldots +w_N ,\quad z\p_t -w_1,\\
&z\dd_{w_1}w_1- t z \p_t -z\dd_{w_i}w_i + w_i +\gamma_{m+i-1}z, \;\text{ for }i=2,\ldots,n,\\
&z\dd_{w_1}w_1 - t z\p_t +z\dd_{w_{n+j}}w_{n+j} - w_{n+j} +\gamma_jz, \;\text{ for }j=1,\ldots,m\, .
\end{align*}

We now consider the zeroth cohomology $\cH^0q_+\left(p^+z\check\cN_A^\gamma\otimes\laa\right)$, which is in turn the $N$-th cohomology of the de Rham complex $q_* \text{DR}_{\mba^1 \times \mbg_m \times \mba^N/\mba^1 \times \mbg_m}\left( p^+ z\check{\mcn}^\gamma_A \otimes \laa \right)$. This is given by the cyclic $\cR^{\text{int}}_{\mba^1 \times \mbg_m}$-module $\cR^{\text{int}}_{\mba^1 \times \mbg_m}/(P',H')$, where the operators $P$ and $H'$ are given by
\[
P':= z^2 \p_z + (n-m) t z\p_t + \varepsilon' z, \quad H':=zt\dd_t\prod_{i=1}^{n-1}(zt\dd_t-\gamma_{m+i}z)-(-1)^mt\prod_{j=1}^m(zt\dd_t-\gamma_jz)
\]
and $\varepsilon':= \sum_{j=1}^m\gamma_j-\sum_{i=m+1}^{N-1}\gamma_i-1$. Replacing $t$ by $(-1)^m t$ we obtain that
$\FL^\psi(z\check{\cN}^\gamma_A)\cong \cR_{\dA^1\times\dG_m}^{\text{int}}/(P',H)$, with
\[
H:=zt\dd_t\prod_{i=1}^{n-1}(zt\dd_t-\gamma_{m+i}z)-t\prod_{j=1}^m(zt\dd_t-\gamma_jz).
\]
Now it follows from Proposition \ref{prop:1dimFL} that
$$
\For_2\left(\FL_{\MTM}^\psi\left({^\cT\!\!}\check\cM_A^\gamma\right)\right)
\cong z^{-N} \FL^\psi(z\check{\cN}^\gamma_A) \cong \cR_{\dA^1\times\dG_m}^{\text{int}}/(P,H)
$$
with $$
P=z^2\partial_z+(n-m)tz\partial_t+\varepsilon z \,\text{ and }\,H=zt\dd_t\prod_{i=1}^{n-1} z(t\partial_t-\gamma_{m+i})-t\prod_{j=1}^m z(t\partial_t-\gamma_j),
$$
and $\varepsilon =\sum_{j=1}^m\gamma_j-\sum_{i=m+1}^{N-1}\gamma_i+N-1$.
\end{proof}

\begin{remark}\label{rem:naiveFT}
As a matter of fact, we do not have to restrict ourselves to the region $\fA_A$ to find our admissible parameters. If we have $\gamma\in\fA_A$ and add to it an integer vector $\underline{k}\in\dZ^{N-1}$ with no negative entries, then $\gamma+\underline{k}\notin sRes(A)$ by definition (cf. the proof of \cite[Lemma 3.5]{RS15}). Therefore, since $\cO_{\dG_m^d}^\gamma\cong\cO_{\dG_m^d}^{\gamma+\underline{k}}$ for any integer vector $\underline{k}$, we have $\check{\cM}_A^\gamma\cong\check{\cM}_A^{\gamma+\underline{k}}$ by Proposition \ref{prop:DirectImage} and the statement of the proposition holds true after changing $\fA_A$ by $\fA_A+\dN^{N-1}$.
\end{remark}

We will also make use of the following result, which calculates the admissible domain $\fA_A$ for the matrix $A$ in our particular context.

\begin{lemma}\label{lem:Raute}
Let $A\in \textup{M}((N-1)\times N,\dZ)$ be the matrix defined at the beginning of the section. Consider a point $p=(p_1,\ldots,p_m,q_1,\ldots,q_{n-1})\in[0,1)^{N-1}$. Let us define
$$p_-:=\min\bigl((\{p_1,\ldots,p_m\}\setminus\{0\})\cup\{1\}\bigr)\,\text{ and }\,p_+:=\max\{p_1,\ldots,p_m\},$$
that is, the minimum of the $p_i$ that do not vanish (taking $p_-=1$ if all of them are zero) and the maximum of them all.

Then, $p$ belongs to $(\fA_A+\dN^{N-1})\subset\dR^{N-1}$ if and only if, for all $i=1,\ldots,n-1$:
\begin{itemize}
\item $q_i\in[0,p_-)$ if some $p_i$ vanishes, or
\item $q_i\in[0,p_-)\cup[p_+,1)$, otherwise.
\end{itemize}
\end{lemma}
\begin{proof}
We will first find the expression for the admissible region $\fA_A$. For this purpose, we must find a set of hyperplanes containing the facets of the cone $C:=\dR_{\geq0}A\subset \mbr^{N-1}$.
Denote by $\{\ul u_1,\ldots,\ul u_{N-1}\}$ the canonical basis of $\dR^{N-1}$ and write $x_1,\ldots,x_{N-1}$ for the corresponding coordinates.

Since any face of a cone is generated by a subset of its generators, and for our given matrix $A$, any $(N-1)\times (N-1)$-minor is non-zero (so that any subset of $N-1$ columns generates a full-dimensional cone), we see that any facet can contain at most $N-2$ columns. On the other hand, such facet must be $N-2$-dimensional, so it cannot be generated by fewer columns. Therefore, we can conclude that it contains exactly $N-2$ columns.

Any linear functional $h$ defining a facet of $C$ must satisfy that $h(C) \geq 0$. Denote by $H_{k,l}$ the hyperplane not containing $\underline{a}_k$ and $\underline{a}_\ell$. There are five classes of these hyperplanes: $H_{1,i}, H_{1,n+j},H_{i_1,i_2},H_{i,n+j},H_{n+j_1,n+j_2}$ with $i,i_1,i_2 \in \{2,\ldots ,n\}$ and $j,j_1,j_2 \in \{ 1,\ldots, m\}$. The linear functionals defining them are, respectively,\begin{align*}
  h_{1,i}&:= x_{m+i-1},\\
  h_{1,n+j}&:= x_j, \\
  h_{i_1,i_2}&:= x_{m+i_1-1}-x_{m+i_2-1}, \\
  h_{i,n+j}&:= x_j-x_{m+i-1},\\
  h_{n+j_1,n+j_2}&:= x_{j_1}-x_{j_2}.
\end{align*}
All of the linear forms $h_{1,i}$, $h_{i_1,i_2}$ and $h_{n+j_1,n+j_2}$ (for the corresponding values of $i,i_1,i_2,j_1,j_2$) take both negative and positive values on some columns of $A$, so the associated hyperplanes do not contain any facet.

We conclude that each facet of $C$ is contained in one of the following hyperplanes:
\begin{equation}\label{eq:Hyperplanes}
\begin{array}{rll}
H_{1,n+j}:&x_j=0 & \text{ for }j=1,\ldots,m,\\
H_{i,n+j}:&x_j-x_{m+i-1}=0 & \text{ for }i=2,\ldots,n, j=1,\ldots,m.
\end{array}
\end{equation}
These hyperplanes are different one from each other and the respective functionals satisfy $h_{1,n+j}(C)\geq 0$ and $h_{i,n+j}(C)\geq 0$. Hence each of them contains a different facet of the cone $C$.

The primitive, inward-pointing normal vectors of the hyperplanes
$H_{1,n+j}$ resp. $H_{i,n+j}$ are $\ul n_{1,n+j}:=\ul u_j$ resp. $\ul n_{i,n+j}:=\ul u_j-\ul u_{m+i-1}$. Denote by $\underline{c}$ the sum of all columns of $A$. We have $\underline{c}=2(\ul u_1+\ldots+\ul u_m)$ and $e_{k,l} :=\langle\ul n_{k,l},\ul c\rangle=2$, where $k$ and $l$ take the admissible values corresponding to the hyperplanes we consider in \eqref{eq:Hyperplanes} (i.e., we have either $(k,l)=(1,n+j)$ or $(k,l)=(i,n+j)$ for $i=2,\ldots,n$ and $j=1,\ldots,m$).
Define
\begin{align*}
\mathfrak{A}_{k,l}&:=H_{k,l} -[0,1/e_{k,l})\cdot\underline{c} = H_{k,l} -[0,1)\cdot\left(\ul u_1+\ldots+\ul u_m\right) \\
&=\begin{cases} H_{1,n+j}-[0,1)\cdot u_j & \text{for } j=1,\ldots,m, \\
H_{i,n+j} -[0,1)\cdot u_j & \text{for } i=2,\ldots,n, j=1,\ldots,m,\end{cases}
\end{align*}
since for $(k,l)=(1,n+j)$ resp. $(k,l)=(i,n+j)$, the vectors $\ul u_1, \ldots, \ul u_{j-1}, \ul u_{j+1},\ldots, \ul u_m$ are contained in $H_{1,n+j}$ resp. $H_{i,n+j}$. Then we have
$$
\mathfrak{A}_{1,n+j}=H_{1,n+j}-[0,1)\cdot u_j = \left\{(x_1,\ldots, x_{N-1})\in\dR^{N-1}\,|\,-1<x_j\leq 0 \right\}
$$
for all $j=1,\ldots, m$ and
$$
\mathfrak{A}_{i,n+j}=H_{i,n+j} -[0,1)\cdot u_j = \left\{(x_1,\ldots, x_{N-1})\in\dR^{N-1}\,|\,-1<x_j-x_{m+i-1}\leq 0\right\}
$$
for all $i=2,\ldots,n, j=1,\ldots,m$.
According to the construction given before Theorem \ref{thm:ClosedImmersion}, we can conclude that
$$
\fA_A =\bigcap_{F \text{ facet}} \{ \mbr \cdot F - [0, 1/e_F) \cdot \underline{c}\}=
\bigcap_{\scriptscriptstyle k,l \textup{ from eq.} \eqref{eq:Hyperplanes}}\fA_{k,l},
$$
so we can describe the admissible region $\fA_A$ as
$$
\fA_A:\begin{cases}
-1<x_j\leq 0 & \text{for }j=1,\ldots,m,\\
-1<x_j-x_{m+i-1}\leq0 & \text{for }i=2,\ldots,n, j=1,\ldots,m
\end{cases}\subset\dR^{N-1}.
$$
Now let us pick a point $p\in[0,1)^{N-1}\cap(\fA_A+\dN^{N-1})$, and take $\underline{k}=(k_1,\ldots,k_{N-1})\in \dN^{N-1}$ such that
$p\in[0,1)^{N-1}\cap(\fA_A+\underline{k})$. The shifted domain is given by
$$
\fA_A+\underline{k}:\begin{cases}
-1+k_j<x_j\leq k_j & \text{for }j=1,\ldots,m,\\
-1+k_j-k_{m+i-1}<x_j-x_{m+i-1}\leq k_j-k_{m+i-1} & \text{for }i=2,\ldots,n, j=1,\ldots,m
\end{cases}\subset\dR^{N-1}.
$$
Assume first there is a vanishing coordinate $p_{j_0}$. Then we must have $k_{j_0}=0$. For such an index and any $i=1,\ldots,n-1$, we can consider the $n-1$ inequalities
$$-1-k_{m+i}<-q_i\leq-k_{m+i},$$
from where we deduce that every $q_i$ belongs to $[k_{m+i},k_{m+i}+1)\cap[0,1)$, for $i=1,\ldots,n-1$. In order for those intersections to be nonempty, we must have $k_{m+i}+1>0$ and $k_{m+i}<1$, so necessarily $k_{m+i}=0$ for all $i$ (and hence $q_i$ must lie within $[0,1)$, which is no new information).

Now, for any nonvanishing $p_j$, it is clear that $k_j=1$. Then, if we look at the remaining inequalities, we see that
$$0<p_j-q_i\leq1,$$
for every $i=1,\ldots,n-1$, and any $j\in\{1,\ldots,m\}$ such that $p_j\neq0$. Therefore, every $q_i$ belongs to $[0,1)\cap\bigcap_{p_j\neq0}[p_j-1,p_j)=[0,p_-)$. Obviously, if $p_j=0$ for all $j=1,\ldots,m$, we obtain that the $q_i$ belong all to $[0,1)=[0,p_-)$.

Assume now that no $p_j$ vanishes. Then $k_1=\ldots=k_m=1$. It follows that we can express the shifted region $\fA_A+\underline{k}$ as
$$
\fA_A+\underline{k}:\begin{cases}
0<x_j\leq 1 & \text{for }j=1,\ldots,m,\\
-k_{m+i-1}<x_j-x_{m+i-1}\leq1-k_{m+i-1} & \text{for }i=2,\ldots,n, j=1,\ldots,m
\end{cases}\subset\dR^{N-1}.
$$
Then, for any $j=1,\ldots,m$, we have $q_i\in[0,1)\cap[p_j+k_{m+i}-1,p_j+k_{m+i})$, for $i=1,\ldots,n-1$. As before, this implies that $p_j+k_{m+i}>0$ and $p_j+k_{m+i}-1<1$, for each $j=1,\ldots,m$. Since each $p_j$ lives in $(0,1)$, the $k_{m+i-1}$ can only be either 0 or 1.

Pick an $i\in\{1,\ldots,n-1\}$ such that $k_{m+i}=0$. Then, as before,
$$q_i\in\bigcap_{j=1}^m[p_j-1,p_j)\cap[0,1)=[0,p_-).$$
If our index $i$ is such that $k_{m+i}=1$, then
$$q_i\in\bigcap_{j=1}^m[p_j,p_j+1)\cap[0,1)=[p_+,1),$$
and one direction of the statement is done.

To show the other implication of the lemma, suppose now that every $q_i$ lies within $[0,p_-)\cup[p_+,1)$ for $i=1,\ldots,n-1$, and no $p_j$ vanishes. We can rewrite this as a disjunction: either $q_i\in\bigcap_{j=1}^m[0,p_j)=[0,1)\cap\bigcap_{j=1}^m[p_j-1,p_j)$ or $q_i\in\bigcap_{j=1}^m[p_j,1)=[0,1)\cap\bigcap_{j=1}^m[p_j,p_j+1)$. If $q_i\in[0,p_-)$, define $k_{m+i}:=0$. Otherwise, we take $k_{m+i}:=1$. Summing up, it is clear that
$$p\in\left(\fA_A+(1,\stackrel{(m)}{\ldots},1,k_{m+1},\ldots,k_{N-1})\right)\cap[0,1)^{N-1}.$$

If some $p_j$ vanishes, and every $q_i$ belongs to $[0,p_-)$, we can do the same as above to see that
$$p\in\left(\fA_A+(k_1,\ldots,k_m,0,\ldots,0)\right)\cap[0,1)^{N-1},$$
where $k_j$ vanishes if so does $p_j$ and is equal to $1$ if $p_j\neq 0$.
\end{proof}

As a consequence of the above calculation of the set of admissible parameters, let us prove a result extending \cite[Theorem 3.24]{CDS}.

\begin{theorem}\label{thm:1dimhypgeo}
Let $\alpha_1,\ldots,\alpha_n$ and $\beta_1,\ldots,\beta_m$ be real numbers, lying on the interval [0,1) and increasingly ordered. Assume moreover that:
\begin{itemize}
  \item No difference $\alpha_i-\beta_j$ is zero, for any $i=1,\ldots,n$ and $j=1,\ldots,m$.
  \item After applying the bijection $[0,1)\ra S^1$ given by $x\mapsto e^{2\pi ix}$, all the images of the $\alpha_i$ are at one arc of the unit circle, while those of the $\beta_j$ find themselves at the complementary arc. (In other words and going back to the interval $[0,1)$, either no $\alpha_i$ belongs to any interval $(\beta_j,\beta_{j+1})$ or viceversa.)
\end{itemize}
Consider the operators $P$ and $H$ given by
$$P=z^2\partial_z+(n-m)tz\partial_t+\varepsilon z \,\text{ and }\,H=\prod_{i=1}^n z(t\partial_t-\alpha_i)-t\prod_{j=1}^m z(t\partial_t-\beta_j),$$
with $\varepsilon =\sum_{j=1}^{m}\beta_j-\sum_{i=1}^n\alpha_i+N-1$. Let $\wh\cH(\alpha_i;\beta_j)$ be the $\cRint_{\dA_z^1\times \dG_m}$-module
$$\widehat{\cH}(\alpha_i;\beta_j):=\cO_{\dA_z^1\times\dG_m}\langle z^2\partial_z,zt\partial_t\rangle/(P,H).$$

Then, $\widehat{\cH}(\alpha_i;\beta_j)$ underlies a unique object of $\IrrMHM(\dG_m)$ with associated $\cD_{\dG_m}$-module $\cH(\alpha_i;\beta_j)$. It can be uniquely extended to an irreducible $\cRint_{\dA^1_z\times \dP^1}$-module underlying an object of $\IrrMHM\left(\dP^1\right)$.
\end{theorem}
\begin{proof}
Let us assume first that $\alpha_1=0$. Then, by the first assumption on the $\alpha_i$ and the $\beta_j$, we have $\beta_j\neq0$ for every $j$. By the second assumption we can deduce that no $\alpha_i$ is between any two $\beta_j$, but all of the $\beta_j$ must be between two certain $\alpha_i$. Thanks to Lemma \ref{lem:Raute}, this means that $(\beta_1,\ldots,\beta_m,\alpha_2,\ldots,\alpha_n)$ belongs to $\fA_A+\dN^{N-1}$, where $A$ is the matrix of the beginning of the section.
As a consequence, by Proposition \ref{prop:naiveFT} and Remark \ref{rem:naiveFT} we have
that
$$
\For_2\left(\FL^\psi_{\MTM}\left({^\cT\!\!}\check\cM_A^\gamma\right)\right)
\cong \wh\cH(\alpha_i;\beta_j)
$$
(recall that ${^\cT\!\!}\check\cM_A^\gamma$ is the algebraic integrable mixed twistor $\cD$-module
with underlying $\cRint_{\dA^N}$-module $\cR^{F^H}\check{\mcm}^\gamma_A$, i.e.
such that $\For_2\left(\check\cM_A^\gamma\right)=\cR^{F^H}\check{\mcm}^\gamma_A$). We have moreover that ${^\cT\!\!}\check\cM_A^\gamma\in \IrrMHM(\dA^N)$ and thanks to \cite[Cor. 0.5]{Sa15}, we know that the functors entering in the definition of $\FL_{\MTM}^\psi$ preserve the category
of irregular mixed Hodge modules, so we conclude that $\wh\cH(\alpha_i;\beta_j)$ underlies an element of
$\IrrMHM(\dG_m)$.

Assume now that $\alpha_1>0$. For any real number $\eta$, denote by $\wh\cK_\eta$ the Kummer $\cR_{\dA^1\times\dG_m}$-module $\cRint_{\dA^1\times \dG_m}/(z^2\dd_z,tz\dd_t-z\eta)$.

The tensor product of $\cRint_{\dA^1\times\dG_m}$-modules $\wh\cH(\alpha_i;\beta_j)\otimes_{\cO_{\dA^1\times\dG_m}}\wh\cK_{-\alpha_1}$ gives rise to the corresponding tensor product of twistor $\cD$-modules on $\dG_m$. This product can be presented as $\wh\cH(\alpha'_i;\beta'_j)$, where $\alpha'_i=\alpha_i-\alpha_1$ for every $i$ and $\beta'_j=\beta_j-\alpha_1$ for every $j$. The assumptions on the parameters imply that $\alpha'_1=0$ and the vector $(\beta'_1,\ldots,\beta'_m,\alpha'_2,\ldots,\alpha'_n)$ lives in $\fA_A+\dN^{N-1}$. Then, arguing as before, such tensor product is an irregular mixed Hodge module of exponential-Hodge origin. Since $\wh\cK_{\alpha_1}$ is the faithful image of a mixed Hodge module on $\dG_m$, the tensor product with it preserves the condition of being in $\IrrMHM(\dG_m)$ due to \cite[Cor. 0.5]{Sa15}, and so is the case of our original $\cRint_{\dA_z^1\times \dG_m}$-module
$$\wh\cH(\alpha_i;\beta_j)\cong\wh\cH(\alpha'_i;\beta'_j) \otimes_{\cO_{\dA^1\times\dG_m}}\wh\cK_{\alpha_1}.$$

This ends the statement on the existence. Let us prove now the claims on the unicity, as in \cite[Thm. 2.13]{CDS}, noting that the condition on the differences $\alpha_i-\beta_j$ is equivalent to $\cH$ being irreducible, and thus rigid (cf. [ibid., Prop. 2.5], noting that all the parameters belong to $[0,1)$).

Consider now any twistor $\cD$-module $\wh\cH'$ on $\dG_{m,t}$ whose underlying $\cD_{\dG_{m,t}}$-module is $\cH$. Since the functor $\Xi_{\text{DR}}$ is faithful by \cite[Rem. 7.2.9]{Mo13}, we have an injection of Hom groups
$$\Hom_{\MTM(\dG_{m,t})}(\wh\cH,\wh\cH')\hra\Hom_{\cD_{\dG_{m,t}}}(\cH,\cH).$$
But $\cH$ is irreducible, so its only endomorphism is the identity and then the twistor $\cD$-module underlying $\cH$ is unique.

On the other hand, let $j:\dG_{m,t}\hra\dP^1$ be the canonical inclusion and consider the $\cD_{\dP^1}$-module $\cH_{pr}:=j_{\dag+}\cH$. It is an irreducible holonomic $\cD_{\dP^1}$-module, because so is $\cH$ by the assumption on the $\alpha_i$ and the $\beta_j$. Then it gives rise to a unique pure integrable twistor $\cD$-module $\wh\cH_{pr}$ on $\dP^1$ by \cite[Thm. 1.4.4]{Mo5} and \cite[Rem. 1.40]{Sa15}. In addition, its underlying $\cD_{\dP^1}$-module $\cH_{pr}$ is rigid, as $\cH$ was. As a consequence, we can invoke [ibid., Thm. 0.7] and claim that such twistor $\cD$-module on $\dP^1$ is in fact an object of $\IrrMHM(\dP^1)$. Take now $\wh\cH':=j^+\wh\cH_{pr}$, which is an irregular mixed Hodge module whose underlying $\cD_{\dG_{m,t}}$-module is $\cH$, by \cite[Prop. 14.1.24]{Mo13}. Then we must have, as was just shown, $\wh{\cH}'\cong \wh{\cH}$, so that the extension $\wh{\cH}_{pr}$ of $\wh{\cH}$ is unique, and we are done.
\end{proof}

\begin{remark}\label{rem:HodgeFiltOrderFilt}
Let us consider the last theorem for the case $m=n$, that is, the case of regular hypergeometric systems.
Consider $\wh\cH$ as a $\cR_{\dA^1_z\times \dG_m}$-module only, as such it is isomorphic to
$\cR_{\dA^1_z\times \dG_m} / (H)$, where now $H=\prod_{i=1}^m z(t\partial_t-\alpha_i)-t\prod_{j=1}^m z(t\partial_t-\beta_j)$.
$\cR_{\dA^1_z\times \dG_m}$ is graded by degree in $z$ (where $z$ has degree $1$), and since $H$ is homogenous (which is
not the case if $n\neq m$), we see that $\wh\cH$ is a \emph{graded} $\cR_{\dA^1_z\times \dG_m}$-module. 
It is obviously \emph{strict}, i.e. it has no $z$-torsion, and then by \cite[A.2.5(5)]{MHM},
we see that $\wh \cH$
is the Rees module of a filtered
$\cD_{\dG_m}$-module, namely, the (regular) hypergeometric module $\cH(\alpha_i;\beta_j)$ together with the filtration by order of differential operators. Notice also that if $n=m$, we have $P=z^2\partial_z+\varepsilon z$, which implies that
$\wh\cH$ has an action by $z\partial_z$ and that if we write $\wh\cH=\oplus_k \wh\cH_k$ (grading with respect to $z$),
then for any $m\in \wh\cH_k$, we have $(z\partial_z)(m)=(k-\epsilon) m$.

Now suppose that we have $n=m$ and that additionally the hypotheses of the last theorem are satisfied, then since $\widehat{\cH}(\alpha_i;\beta_j)$ is the unique object in $\IrrMHM(\dG_m)$ (lying actually in the essential image of $\MHM(\dG_m)$) with underlying $\cD_{\dG_m}$-module $\cH(\alpha_i;\beta_j)$, it is the Rees module of the filtered module $(\cH(\alpha_i;\beta_j),F^H_\bullet)$, where
$F^H_\bullet$ denotes the Hodge filtration of the complex variation of Hodge structures on $\cH(\alpha_i;\beta_j)$.
Hence $F^H_\bullet \cH(\alpha_i;\beta_j) = F^{ord}_\bullet\cH(\alpha_i;\beta_j)$ in this case.
Moreover, if we put
$$
R_k:=\prod_{i=1}^{k} (t\partial_t-\alpha_i)
$$
for $k=0,\ldots,n-1$ (where $R_0:=1$), then $(R_k)_{k=0,\ldots,n-1}$ is an $\cO_{\dG_m}$-basis of $\cH(\alpha_i,\beta_j)$ and
yields a splitting of the Hodge filtration $F_\bullet^H$. In particular, we obtain that the Hodge numbers $h^p(\cH(\alpha_i;\beta_j))
=\dim\left(F_k^H /F_{k-1}^H\right)$
are all equal to one. This is consistent with \cite[Thm. 1]{Fe} (up to an overall shift, as noticed in that theorem)
in the version of \cite[Proposition 2.6]{CDS},
since under the assumption of Theorem \ref{thm:1dimhypgeo}, the function $\#\{j:\beta_j<\alpha_k\}$ is constant.
\end{remark}

We will finish this section with a calculation of an irregular Hodge filtration, similar to the last section of \cite{CDS}. In that reference, the authors computed such a filtration in the case where the hypergeometric $\cD$-module had a purely irregular singularity at infinity, that is, it was of type $(n,0)$. It is immediate to see that for modules of type $(n,1)$, the second assumption of Theorem \ref{thm:1dimhypgeo} holds true, so that we obtain an explicit description of the $\cRint_{\dA^1_z\times \dG_m}$-module underlying the irregular Hodge module with associated $\cD_{\dG_m}$-module $\cH(\alpha_1,\ldots,\alpha_n;\beta)$. In the sequel, we are going to compute the irregular Hodge filtration of such modules of type $(n,1)$.

Let us recall the conventions and notations used in \cite[\S~4]{CDS} (cf. \cite[Not. 2.1]{Sa15}). We will deal with the classical hypergeometric $\cD$-module $\cH=\cH(\alpha_i;\beta)$, where the $\alpha_i$ and $\beta$ are $n+1$ real numbers belonging to the interval $[0,1)$. We will denote by $\wh\cH$ both its associated algebraic, integrable twistor $\cD$-module on $\dG_m$ and its underlying $\cR_{\dA_z^1\times\dG_m}^{\text{int}}$-module (as in the statement of Theorem \ref{thm:1dimhypgeo}). From now on, we will write $\cX$, $\ttheta\cX$ and $\ttau\cX$ meaning the products $\dA_z^1\times \dG_{m,t}$, $\cX\times\dG_{m,\theta}$, and $\cX\times\dA_\tau^1$, respectively, where $\theta=1/\tau$. Finally, we will write $\ttau\cX_0=\cX\times\{\tau=0\}\subset\ttau\cX$.

\begin{thm}
Let real numbers $\alpha_1,\ldots,\alpha_n, \beta \in [0,1)$ be given. Suppose that $\alpha_1\leq \ldots \leq \alpha_n$ and that moreover $\alpha_i-\beta \notin \dZ$ for all $i=1,\ldots,n$. For each $k=1,\ldots,n$, set $\rho(k)=-(n-1)\alpha_k+k$. Then the jumping numbers of the irregular Hodge filtration of $\cH=\cH(\alpha_i;\beta)$ are, up to an overall real shift, the numbers $\rho(k)$. The irregular Hodge numbers are the multiplicities of those jumping numbers, or equivalently, the nonzero values of $|\rho^{-1}(x)|$, for $x$ real.

Moreover, let $\nu_\alpha(k)=\lceil -\alpha+k-\ep-(n-1)\alpha_{k+1}\rceil$ (recall
from Theorem \ref{thm:1dimhypgeo} that $\ep=\beta-\sum_{i=1}^{n} \alpha_i+n$). Let us consider the operators
$$\bar{Q}_k=(-(n-1))^k\prod_{i=1}^k(t\dd_t-\alpha_i)$$
for $k=0,\ldots,n-2$ (where the empty product equals one) and
$$\bar{Q}_{n-1}=(-(n-1))^{n-1}\prod_{i=1}^{n-1}(t\dd_t-\alpha_i)+ \frac{(-(n-1))^{n-1}t(\beta-\alpha_1)}{1+\alpha_1-\alpha_n}\bar{Q}_0.$$
Then, the irregular Hodge filtration $\Firr_\bullet\cH$ is given by
$$\Firr_{\alpha+j}\cH=\bigoplus_{k:j\geq\nu_\alpha(k)}\cO_X\bar{Q}_k.$$
\end{thm}

\begin{rem}

In general, the procedure given below can be of use to find an explicit expression for the irregular Hodge filtration, not only the numbers, of any hypergeometric of type $(n,m)$, provided both assumptions from Theorem \ref{thm:1dimhypgeo} are fulfilled. However, the calculations become soon too cumbersome to be included here.
\end{rem}

\begin{proof}
We will mimic the arguments of \cite[\S~4]{CDS}, providing almost no proof of the claims which are similar to some therein.

We must first consider the rescaling of $\wh\cH$: this is the inverse image $\ttheta\wh\cH:=\mu^*\cH$ (as $\cO_{\ttheta\cX}$-module), endowed with a natural action of $\cRint_{\ttheta\cX}$ as depicted in \cite[2.4]{Sa15} (note that $\theta=\tau^{-1}$), where $\mu$ is the morphism given in [ibid., Not. 2.1] by
$$\begin{array}{rrcl}
\mu:&\ttheta\cX&\ra&\cX\\
&(z,t,\theta)&\mapsto&(z\theta,t).
\end{array}$$
In this sense, we can apply the same argument of \cite[Prop. 4.1]{CDS} to get that the $\cRint_{\ttheta\cX}$-module $\ttheta\wh\cH$ associated with $\wh\cH$ can be presented as $\cRint_{\ttheta\cX}/(P,\ttheta R,\ttheta H)$, where $P=z^2\dd_z+(n-m)tz\dd_t+\ep z$ as in Theorem \ref{thm:1dimhypgeo}, $\ttheta R=z^2\dd_z-z\theta\dd_\theta$ and
$$\ttheta H=\prod_{i=1}^nz\theta(t\dd_t-\alpha_i)-tz\theta(t\dd_t-\beta).$$
Now we have to invert $\theta$ to obtain an $\cRint_{\ttau\cX}(*\ttau \cX_0)$-module $\ttau\wh\cH$, to work in the setting given by \cite[\S 2.3]{Sa15}. In this sense, we will denote by $\ttau\wh\cH$ the $\cRint_{\ttau\cX}(*\ttau\cX_0)$-module $(\id_{\cX}\times(j\circ\inv))_*\ttheta\wh\cH$, where $\inv:\mathbb{G}_{m,\theta}\rightarrow\mathbb{G}_{m,\tau}$ is the inversion operator $\theta\mapsto\tau$ and $j:\mathbb{G}_{m,\tau}\hookrightarrow\dA_{\tau}^1$ is the canonical inclusion. Then it is easy to see that $\ttau\wh\cH=\cRint_{\ttau\cX}(*\ttau \cX_0)/(P,\ttau R,\ttau H)$, with $P$ as always, $\ttau R=z^2\dd_z+z\tau\dd_\tau$ and
$$\ttau H=\prod_{i=1}^n\frac{z}{\tau}(t\dd_t-\alpha_i)-t\frac{z}{\tau}(t\dd_t-\beta).$$

The next step is forming the basis of $\ttau\wh\cH$ as a $\cO_{\ttau\cX}(*\ttau \cX_0)$-module. Let it be given by
$$Q_k=(-(n-1))^k\prod_{i=1}^k\frac{z}{\tau}(t\dd_t-\alpha_i)$$
for $i=0,\ldots,n-2$ and
$$Q_{n-1}=(-(n-1))^{n-1}\prod_{i=1}^{n-1}\frac{z}{\tau}(t\dd_t-\alpha_i)+ \frac{(-(n-1))^{n-1}t(\beta-\alpha_1)}{1+\alpha_1-\alpha_n}Q_0.$$
It is indeed a basis: we can use the expressions of $\ttau R$ and $P$ to replace the classes of $z\tau\dd_\tau$ and $z^2\dd_z$, respectively, in terms of $zt\dd_t$. Now $\ttau\wh\cH$ is generated as a $\cO_{\ttau\cX}(*\ttau \cX_0)$-module by the powers of $zt\dd_t$, and we can get rid of those of exponent greater than $n-1$ using $\ttau H$. The remaining $n$ powers can be expressed as a linear combination of the $Q_i$, forming a triangular matrix (almost diagonal in fact), so the latter conform a basis as well.

One could wonder about the odd expression of the $Q_i$. In the case with no betas of \cite{CDS}, the basis considered there was formed just by the successive products $\prod_{i=1}^k\frac{z}{\tau}(t\dd_t-\alpha_i)$, up to some constant. In this case, such a basis does not provide a connection matrix solving the Birkhoff problem with a diagonal matrix as a coefficient of the pole at infinity in $z$, which would give us a way to read the spectrum from that matrix (cf. \cite[Prop. 4.8]{dGMS}). As a consequence, we have to adapt such initial basis, and that is how we get the $Q_i$. Let us write the connection matrix explicitly.

Let $c=(\beta-\alpha_1)/(1+\alpha_1+\alpha_n)$, in such a way that
$$Q_{n-1}=(-(n-1))^{n-1}\prod_{i=1}^{n-1}\frac{z}{\tau}(t\dd_t-\alpha_i)+(-(n-1))^{n-1}ctQ_0.$$
A similar (but longer) calculation to the proof of \cite[Lem. 4.3]{CDS} shows that the integrable connection arising from the $\cRint_{\ttau\cX}(*\ttau X_0)$-module structure associated with $\ttau\wh\cH$ has the following matrix form:
$$\nabla\underline Q=\underline Q\left(\left(\tau A_0+zA_\infty\right)\frac{dz}{z^2}+\left(-\tau A_0+zA'_\infty\right)\frac{dt}{(n-1)zt}-\left(\tau A_0+zA_\infty\right)\frac{d\tau}{z\tau}\right).$$
There, if $n>2$, $A_0$, $A'_\infty$ and $A_\infty$ are the matrices
\begin{equation}\label{eq:MatrixForm}\begin{gathered}
A_0=\begin{pmatrix}
0&\cdots&-(-(n-1))^{n-1}ct & 0\\
1&\ddots& & (-(n-1))^{n-1}(c+1)t\\
 &\ddots&0&\vdots\\
 & & 1& 0\end{pmatrix},\\
A'_\infty=\diag((n-1)\alpha_1,\ldots,(n-1)\alpha_n)\text{ and }A_\infty=\diag(0,1,\ldots,n-1)-\ep I_n-A'_\infty.
\end{gathered}\end{equation}
If $n=2$, we have
\begin{equation}\label{eq:MatrixForm2}
A_0=\begin{pmatrix}
ct&c(c+1)t^2\\
1&(c+1)t
\end{pmatrix},\,
A'_{\infty}=\begin{pmatrix}
\alpha_1&0\\
0&\alpha_2
\end{pmatrix}
\text{ and } A_\infty=\diag(0,1)-\ep I_2-A'_\infty.
\end{equation}

Finally, the irregular Hodge filtration is obtained from a suitable $V$-filtration along the divisor $\tau=0$ defined on $\ttau\wh\cH$, which is called $\ttau V$-filtration (the new symbol $\ttau V$ is to make clear the variety over which we are working; note the same convention in \cite{Sa15}, from Remark 2.20 on). We are actually to define a filtration on $\ttau\wh\cH$, and then prove that it equals the $\ttau V$-filtration, following \cite[\S 2.1.2]{Mo13}.

Let us consider then
\begin{equation}\label{eq:filtrations}\begin{split}
\ttau U_\alpha\ttau\wh\cH:=\left\{\sum_{k=0}^{n-1}f_k\tau^{\nu_{k}}Q_k\,:\, f_k\in\cO_{\ttau\cX}\text{ , }\,\max(k-(n-1)\alpha_{k+1}-\ep-\nu_k)\leq\alpha\right\},\\
\ttau U_{<\alpha}\ttau\wh\cH:=\left\{\sum_{k=0}^{n-1}f_k\tau^{\nu_k}Q_k\,:\, f_k\in\cO_{\ttau\cX}\text{ , }\,\max(k-(n-1)\alpha_{k+1}-\ep-\nu_k)<\alpha\right\},
\end{split}\end{equation}
for any $\alpha\in\dR$.

The $\ttau U_\alpha\ttau\wh\cH$ form an increasing filtration, indexed by the real numbers but with a discrete set of jumping numbers, such that $\tau\ttau U_\alpha\ttau\wh\cH=\ttau U_{\alpha-1}\ttau\wh\cH$ for any $\alpha$ (those are conditions i and ii' in \cite[\S~2.1.2]{Mo13}). As usual, the graded piece associated with $\alpha$ is $\Gr_\alpha^{\ttau U}\ttau\wh\cH=\ttau U_\alpha\ttau\wh\cH/\ttau U_{<\alpha}\ttau\wh\cH$.

In \eqref{eq:filtrations}, all the exponents $\nu_k$ of the powers of $\tau$ accompanying the $f_kQ_k$ satisfy that $\nu_k\geq-\alpha+k-(n-1)\alpha_{k+1}-\ep$. Then we can define the steps of the filtration in the same alternative way as in \cite[Rem. 4.5]{CDS} as the free $\cO_{\ttau\cX}$-modules of finite rank
\begin{equation}\label{eq:AlternativeFiltration}
\ttau U_\alpha\ttau\wh\cH=\bigoplus_{k=0}^{n-1}\cO_{\ttau\cX}\cdot\tau^{\nu_{\alpha}(k)}Q_k,
\end{equation}
where $\nu_\alpha(k)=\lceil-\alpha+k-\ep-(n-1)\alpha_{k+1}\rceil$. With that expression, it is clear that the graded pieces $\text{Gr}_\alpha^{\ttau U}\ttau\wh\cH$ are
$$\text{Gr}_\alpha^{\ttau U}\ttau\wh\cH=\bigoplus_{k=0}^{n-1}\cO_{\cX}\cdot\tau^{\nu_{\alpha}(k)}Q_k,$$
which are strict $\cR_{\cX}$-modules (condition iv in \cite[\S~2.1.2]{Mo13}).

The next step in the proof is proving that $\ttau\wh\cH$ is strictly $\dR$-specializable along $\ttau \cX_0$ and its $\ttau V$-filtration is actually given by the $\ttau U_\alpha\ttau\wh\cH$. Although the proof is similar to that of \cite[Prop. 4.6]{CDS}, we have to adapt it a bit to our case here.

After what we already showed, it remains to show conditions iii' and v of \cite[\S~2.1.2]{Mo13} and prove that the $\ttau U_\alpha\ttau\wh\cH$ are coherent $V_0\cR_{\cX}$-modules. Let us start by the second condition. Consider then the mappings $\fp,\fe$ given by
$$\begin{array}{rcl}
(\fp,\fe):\dR\times\dC&\longrightarrow&\dR\times\dC\\
(\beta,\omega)&\longmapsto&(\beta+2\Re(z\bar\omega),-\beta z+\omega-\bar{\omega}z^2)\end{array}.$$
We must check that the operator $z\tau\dd_\tau-\fe(\beta,\omega)$ is nilpotent on the graded pieces $\Gr_\alpha^{\ttau U}\ttau\wh\cH$ only for a finite amount of $(\beta,\omega)\in\cK:=\{\beta+2\Re(z_0\bar\omega)=\alpha\}$, for any value $z_0$ of $z$. Moreover, those $(\beta,\omega)$ should belong in fact to $\dR\times\{0\}$ (cf. \cite[\S 1.3.a]{Sa15}), if we want to obtain the $\dR$-specializability.

Take then $(\beta,\omega)\in\cK$ and $f\tau^\nu Q_k\in\ttau U_\alpha\ttau\wh\cH$, with $f\in\cO_{\ttau\cX}$. We must have that $k-(n-1)\alpha_{k+1}-\ep-\nu\leq\alpha$. Assume that $n>2$ and $k<n-2$. Thanks to the matrix form \eqref{eq:MatrixForm} we know that
$$(z\tau\dd_\tau-\fe(\beta,\omega))f\tau^\nu Q_k=\big(z\tau\dd_\tau+(\nu+(n-1)\alpha_{k+1}+\ep-k+\beta)z-\omega+\bar{\omega}z^2\big)(f)\tau^\nu Q_k-f\tau^{\nu+1}Q_{k+1}.$$
Recall that the $\alpha_i$ are increasingly ordered, lying within the interval $[0,1)$. Thus $f\tau^{\nu+1}Q_{k+1}$ lives in $\ttau U_\alpha\ttau\wh\cH$, for
$$k+1-(n-1)\alpha_{k+2}-\ep-\nu-1\leq\left((k+1)-(n-1)\alpha_{k+2}-\ep)- (k-n\alpha_{k+1}-\ep)\right)-1+\alpha\leq\alpha.$$
Now we should look at what happens to the class of $f\tau^{\nu+1}Q_{k+1}$ in the $\alpha$-graded piece of $\ttau\wh\cH$.

Note that $\left[f\tau^\nu Q_k\right]\neq0$ if and only if $\nu+(n-1)\alpha_{k+1}+\ep-k+\alpha=0$, so
$$(z\tau\dd_\tau-\fe(\beta,\omega))f\tau^\nu Q_k=\big(z\tau\dd_\tau+(\beta-\alpha)z-\omega+\bar{\omega}z^2\big)(f)\tau^\nu Q_k-f\tau^{\nu+1}Q_{k+1}=$$
$$=\big(z\tau\dd_\tau-2\Re(z_0\bar\omega)z-\omega+\bar{\omega}z^2\big)(f)\tau^\nu Q_k-f\tau^{\nu+1}Q_{k+1}.$$

Now notice that $\tau$ divides $\tau\dd_\tau(f)$, so in fact $z\tau\dd_\tau(f)\tau^\nu Q_k\in\ttau U_{\alpha-1}\ttau\wh\cH$ and then we can further reduce our expression to
$$(z\tau\dd_\tau-\fe(\beta,\omega))f\tau^\nu Q_k=(-\omega-2\Re(z_0\bar\omega)z+\bar{\omega}z^2)f\tau^\nu Q_k-f\tau^{\nu+1}Q_{k+1}.$$

On the other hand, $\tau^{\nu+1}Q_{k+1}$ does not vanish either in $\Gr_\alpha^{\ttau U} \ttau\wh\cH$ if and only if $\alpha_{k+2}=\alpha_{k+1}$. Indeed, we know that $\nu+(n-1)\alpha_{k+1}+\ep-k+\alpha=0$, so doing the same as before, $k+1-(n-1)\alpha_{k+2}-\ep-\nu-1=\alpha+(n-1)(\alpha_{k+2}-\alpha_{k+1})$ and the claim follows. Furthermore, in order to $(z\tau\dd_\tau-\fe(\beta,\omega))$ to vanish, we should impose that $\omega=0$, just by looking at the coefficients of the powers of $z$ in the expression for $f$.

If $k=n-2$, we obtain from \ref{eq:MatrixForm} that
\begin{align*}
(z\tau\dd_\tau-\fe(\beta,\omega))f\tau^\nu Q_{n-2}&=\big(z\tau\dd_\tau+(\nu+(n-1)\alpha_{n-1}+\ep-(n-2)+\beta)z-\omega+\bar{\omega}z^2\big)(f)\tau^\nu Q_{n-2}\\
&-f\tau^{\nu+1}Q_{n-1}+f\tau^{\nu+1}(-(n-1))^{n-1}ctQ_0.
\end{align*}

Since $-(n-1)\alpha_1-\ep-\nu-1\leq-(n-1)(\alpha_1-\alpha_{n-1}+1)+\alpha<\alpha$ because $\alpha_{n-1}<\alpha_1+1$, the last summand above belongs to $\ttau U_{<\alpha}\ttau\wh\cH$, and then the argument can follow as with $k<n-2$.

Now if $k=n-1$, then everything would be the same again as before except we get the additional summand $-\tau^{\nu+1}Q_{k+1}$, which becomes $-f\tau^{\nu+1}(-(n-1))^{n-1}(c+1)tQ_1$, whose class vanishes in the graded piece under consideration, too. Indeed,
$$1-(n-1)\alpha_2-\ep-\nu-1\leq-(n-1)(\alpha_2-\alpha_n+1)+\alpha<\alpha,$$
for $\alpha_n<\alpha_2+1$.

In conclusion, $(z\tau\dd_\tau-\fe(\beta,\omega))^lf\tau^\nu Q_k$ can only vanish in $\Gr_\alpha^{\ttau U}\ttau\wh\cH$ if $\alpha=\beta$ (and then $\omega=0$), and does not do so until we get to an index $k+l$ such that $\alpha_{k+l}$ is strictly bigger than $\alpha_k$. Since there is a finite set of indexes, $(z\tau\dd_\tau-\fe(\beta,\omega))$ is nilpotent, of nilpotency index $n$ at most.

When $n=2$, we notice from \eqref{eq:MatrixForm2} that we have two possibilities. If $k=0$, everything is the same as with $k=n-2$ for $n>2$, and if $k=1$,
\begin{align*}
(z\tau\dd_\tau-\fe(\beta,\omega))f\tau^\nu Q_1&=\big(z\tau\dd_\tau+(\nu+\alpha_2+\ep-1+\beta)z-\omega+\bar{\omega}z^2\big)(f)\tau^\nu Q_1\\
&+f\tau^{\nu+1}(c+1)tQ_1+f\tau^{\nu+1}c(c+1)t^2Q_0.
\end{align*}
Here the argument runs similarly as in the general case.

Condition iii' can be rephrased as $z\tau\dd_\tau\ttau U_\alpha\ttau\wh\cH\subseteq\ttau U_\alpha\ttau\wh\cH$, using that $\ttau U_\alpha\ttau\wh\cH=\tau\ttau U_{\alpha+1}\ttau\wh\cH$, and that follows essentially from the same argument used to prove condition v above. Last, since $V_0\cR_{\cX}=\cO_{\ttau\cX}\langle z\dd_t,z\tau\dd_\tau\rangle$, it is clear from the computations above and the alternative expression \eqref{eq:AlternativeFiltration} for the filtration steps that they are cyclic $V_0\cR_{\cX}$-modules, and then coherent. Summing up and noting that all the calculations performed were in fact independent of $z_0$, $\ttau\wh\cH$ is strictly $\dR$-specializable along $\ttau \cX_0$ and the $\ttau U_\bullet\ttau\wh\cH$ form its $\ttau V$-filtration.

We can finally show the expression for the irregular Hodge filtration and then the irregular Hodge numbers like in \cite[Thm. 4.7]{CDS}. Since we know that $\wh\cH$ underlies an object in $\IrrMHM(\dG_{m,t})$  by Theorem \ref{thm:1dimhypgeo},
we deduce by \cite[Def. 2.52]{Sa15} that $\wh\cH$ is well-rescalable (cf. [ibid., Def. 2.19]) and so we can apply [ibid., Def. 2.22]. After formula \eqref{eq:AlternativeFiltration}, we clearly have
$$i_{\tau=z}^*\ttau V_\alpha\ttau\wh\cH=\ttau V_\alpha\ttau\wh\cH/(\tau-z)\ttau V_\alpha\ttau\wh\cH=\bigoplus_k\cO_{\cX}z^{\nu_\alpha(k)}\bar{Q}_k,$$
which is free $z$-graded of finite rank. Denote by $\pi$ the projection $\cX\ra \dG_{m,t}$. Then, the $z$-adic filtration on $\pi^*\cH[z^{-1}]$ induces a filtration on $i_{\tau=z}^*\ttau V_\alpha\ttau\wh\cH$, given by
$$F_ri_{\tau=z}^*\ttau V_\alpha\ttau\wh\cH:=\bigoplus_{s\leq r}\left(\bigoplus_{k\,:\,\nu_\alpha(k)\leq s}\cO_{\dG_{m,t}}\bar{Q}_k\right)z^s.$$
Then, $\Gr^F\left(i_{\tau=z}^*\ttau U_\alpha\ttau\wh\cH\right)$ is the Rees module associated to a new good filtration $\Firr_{\alpha+\bullet}\cH$ on $\cH$, for some $k=0,\ldots,n-1$, which is the irregular Hodge filtration. More concretely, $\Firr_\bullet\cH$ is given by
$$\Firr_{\alpha+j}\cH=\bigoplus_{k\,:\,\nu_\alpha(k)\leq j}\cO_{\dG_{m,t}}\bar{Q}_k.$$
Therefore, its jumping numbers are $-\ep+j-1-(n-1)\alpha_j$ for $j=1,\ldots,n$. Since the irregular Hodge filtration is defined up to an overall real shift, we can normalize the jumping numbers to $j-(n-1)\alpha_j$ and the irregular Hodge numbers will be their multiplicities.
\end{proof}

\bibliographystyle{amsalpha}

\def\cprime{$'$}
\providecommand{\bysame}{\leavevmode\hbox to3em{\hrulefill}\thinspace}
\providecommand{\MR}{\relax\ifhmode\unskip\space\fi MR }
\providecommand{\MRhref}[2]{  \href{http://www.ams.org/mathscinet-getitem?mr=#1}{#2}
}
\providecommand{\href}[2]{#2}

\end{document}